\documentclass[11pt,leqno]{article}

\usepackage{amsfonts,latexsym,amsmath,amssymb,amsthm}
\usepackage{hyperref}
\usepackage{fullpage}

\newtheorem{theorem}{Theorem}[section]
\newtheorem{lemma}[theorem]{Lemma}
\newtheorem{proposition}[theorem]{Proposition}

\newtheorem{remark}[theorem]{Remark}
\newtheorem{claim}[theorem]{Claim}

\theoremstyle{definition}
\newtheorem{definition}[theorem]{Definition}

\numberwithin{equation}{section}

\newcommand{\ls}{\leqslant}
\newcommand{\gr}{\geqslant}
\newcommand{\R}{\mathbb{R}}

\DeclareMathOperator{\vol}{vol}

\usepackage{setspace}
\begin{document}
\small

\title{\bf Banach--Mazur distances and basis constants \\ of isotropic log-concave random spaces}
\author{Apostolos Giannopoulos and Antonios Hmadi}
\date{}
\maketitle

\begin{abstract}
\footnotesize
We study the Banach--Mazur distance between random normed spaces generated by centrally symmetric random polytopes associated with isotropic log-concave measures in $\mathbb{R}^n$. 
We show that, in a wide range of parameters, if $x_1,\dots,x_m$ and $y_1,\dots,y_m$ are independent samples from an isotropic log-concave probability measure on $\mathbb{R}^n$, then the corresponding normed spaces $X_{B_m}$ and $Y_{A_m}$ generated by their absolute convex hulls satisfy, with high probability,
$$d_{{\rm BM}}(X_{B_m},Y_{A_m}) \gr \frac{cn}{\ln(1+m/n)},$$
which is sharp in both $n$ and $m$ and recovers the extremal order $n$ when $m \approx n$.

Our results extend Gluskin's theorem from the Gaussian setting to general isotropic log-concave measures, providing evidence for a universality phenomenon in the extremal geometry of the Banach--Mazur compactum. 
In addition, we investigate operator-theoretic properties of the associated random spaces and, as consequences, we derive sharp estimates for their basis constant and show that these random spaces are far from the class of spaces with a $1$-unconditional basis.
The proofs combine probabilistic and geometric methods with recent advances related to Bourgain's slicing problem.
\end{abstract}

\section{Introduction}\label{section:1}

The Banach--Mazur compactum plays a central role in the local theory of Banach spaces and understanding its metric structure is a fundamental problem in asymptotic geometric analysis. 
A landmark result in this direction is Gluskin's theorem~\cite{Gluskin-1981}, which shows that the diameter of the $n$-dimensional Banach--Mazur compactum is of order $n$. 
More precisely, there exist $n$-dimensional normed spaces $X_n$ and $Y_n$ such that
$$d_{{\rm BM}}(X_n,Y_n) \gr cn$$
where $d_{{\rm BM}}(\cdot,\cdot)$ denotes the Banach--Mazur distance and $c>0$ is an absolute constant. 
Gluskin's proof introduced a model of random normed spaces generated by taking the absolute convex hull of the canonical basis together with a collection of independent random vectors distributed uniformly on the Euclidean sphere (or, equivalently, Gaussian vectors). 
This construction has since become a fundamental tool in the field.

A natural question is whether this extremal behavior is specific to Gaussian-type randomness or whether it reflects a more general geometric principle. 
In particular, one may ask whether similar extremal Banach--Mazur distances arise for random polytopes generated by isotropic log-concave measures, which form a broad and fundamental class in asymptotic convex geometry.

\medskip 

\noindent\textbf{Informal statement.} 
Our main goal is to show that the extremal phenomenon discovered by Gluskin persists in this general setting: two independent random normed spaces generated from samples of an arbitrary isotropic log-concave distribution are, with high probability, almost maximally far apart in the Banach--Mazur metric, up to the optimal logarithmic correction. 
This provides evidence for a universality principle governing extremal distances in the Banach--Mazur compactum.

A key new ingredient in our approach is the recent affirmative resolution of Bourgain's slicing problem, which yields uniform control of isotropic constants and allows us to extend probabilistic and geometric arguments beyond the Gaussian framework.

\medskip

\noindent\textbf{Main results.} 
Let $x_1,\dots,x_m$ and $y_1,\dots,y_m$ be independent random vectors with isotropic log-concave distribution $\mu$ on $\mathbb{R}^n$. 
We consider two models of centrally symmetric random polytopes:

\smallskip 

\noindent \emph{Basis-enriched model:} 
Let $2n\ls m\ls\exp(\sqrt{n})$. 
We define
$$B_m := {\rm absconv}\left\{e_1,\ldots,e_n, \frac{x_1}{\sqrt{n}},\ldots,\frac{x_m}{\sqrt{n}}\right\},\quad A_m := {\rm absconv}\left\{e_1,\ldots,e_n, \frac{y_1}{\sqrt{n}},\ldots,\frac{y_m}{\sqrt{n}}\right\}$$
where $\{e_1,\ldots ,e_n\}$ is the standard orthonormal basis of $\mathbb{R}^n$ and denote by $X_{B_m}$ and $Y_{A_m}$ the corresponding normed spaces. 

\smallskip 

\noindent \emph{Pure random polytope model:} 
Let $Cn\ls m\ls\exp(\sqrt{n})$ where $C>0$ is an absolute constant. 
We define
$$B_m := {\rm absconv}\left\{x_1,\ldots,x_m\right\},\quad A_m := {\rm absconv}\left\{y_1,\ldots,y_m\right\}$$
and denote by $X_{B_m}$ and $Y_{A_m}$ the corresponding normed spaces. 

\smallskip 

For both models, in Section~\ref{section:4} we prove the following theorem.

\begin{theorem}\label{th:1.1}
Let $\mu$ be an isotropic log-concave probability measure on $\mathbb{R}^n$. 
If $B_m$ and $A_m$ denote the random polytopes of either the basis-enriched or the pure random polytope model, then for every $m$ in the corresponding range, with probability tending to $1$ as $n \to \infty$, uniformly in $m$, we have
$$d_{{\rm BM}}(X_{B_m}, Y_{A_m}) \gr \frac{cn}{\ln(1+m/n)},$$
where $c>0$ is an absolute constant.
\end{theorem}

The estimate is sharp in its dependence on both $n$ and $m$. 
The logarithmic term is unavoidable for $m \gg n$, reflecting the geometric fact that random polytopes become increasingly well-rounded and approximate the Euclidean ball. 
In the extremal regime $m \approx n$, we recover Gluskin's optimal lower bound of order $n$.

\smallskip

Gluskin's theorem was the starting point for a systematic study of random spaces.
An important feature of the class $X_{n,m}$ of random spaces with unit ball
$$B_m={\rm absconv}\{e_1,\ldots ,e_n,x_1,\ldots ,x_m\}$$
where $x_1,\ldots ,x_m$ are uniformly distributed on the unit sphere $S^{n-1}$, is that a random space $X\in X_{n,m}$ has a rather poor family of well-bounded operators. 
It was observed by Gluskin in \cite{Gluskin-1981b} that a random $X\in X_{n,n^2}$ has the following property: any projection $P$ in $X$ of rank $k\ls n/2$ satisfies
$$\| P:X\to X\|\gr ck/\sqrt{n\ln n}.$$
As a consequence, such a space has basis constant ${\rm bc}(X)\gr c\sqrt{n/\ln n}$ (see Section~\ref{section:2} for background information).

Szarek \cite{Szarek-1983} modified the random structure of $X_{n,m}$ and was able to construct an $n$-dimensional normed space $X$ with ${\rm bc}(X)\gr c\sqrt{n}$. 
John's theorem shows that this order is optimal. 
Mankiewicz \cite{Mankiewicz-1984} applied the random spaces method to construct finite-dimensional spaces with the worst possible symmetric constant (in order). 
In this work, Mankiewicz used the ``space mixing" property of the irreducible group of operators. 
Szarek explicitly introduced the notion of the class ${\rm Mix}(k,\beta )$ of mixing operators, which is the set of all linear operators $T$ satisfying
$${\rm dist}(T(x),E)=| P_{E^{\perp }}T(x)| \gr \beta |x|$$
for some $k$-dimensional subspace $E$ and every $x\in E$. 
Then Szarek showed in \cite{Szarek-1986a} that the mixing property is sufficient for proving the results of Mankiewicz \cite{Mankiewicz-1984}, but also those of Gluskin \cite{Gluskin-1981b} and Szarek \cite{Szarek-1983}.
In particular, he proved that for a random space $X\in X_{n,n^2}$ one has
$$\| T:X\to X\|\gr c\beta k/\sqrt{n\ln n}$$
for any $T\in {\rm Mix}(k,\beta )$.

In Section~\ref{section:6}, using this technique we show that if $m\gr Cn^2(\ln n)$ then, for every $n/4\ls k\ls n/2$ and every $T\in {\rm Mix}_n(k,1)$, a log-concave random space $X_{B_m}$ of the pure random polytope model satisfies
$$\|T:X_{B_m}\to X_{B_m}\|\gr \frac{c_1\sqrt{n}}{\sqrt{\ln(1+m/n)}}$$
with probability greater than $1-\exp(-c_2\sqrt{n})$. 
It is not difficult to show that any projection $P$ of rank  $k\ls n/2$ is $(k,1/2)$-mixing. 
As a consequence, we obtain the following theorem.

\begin{theorem}\label{th:1.2}
Let $\mu$ be an isotropic log-concave probability measure on $\mathbb{R}^n$. 
For any $m\gr Cn^2(\ln n)$, a random polytope $B_m$ of the pure random polytope model satisfies, with high probability,
$$\|P:X_{B_m}\to X_{B_m}\|\gr\frac{c\sqrt{n}}{\sqrt{\ln(1+m/n)}}$$
for every projection $P:\mathbb{R}^n\to\mathbb{R}^n$ with rank $n/4\ls {\rm rank}(P)\ls 3n/4$. 
In particular, for a random $B_m$ we have ${\rm bc}(X_{B_m})\gr c\sqrt{n}/\sqrt{\ln(1+m/n)}$.
\end{theorem}

In Section~\ref{section:5} we show that log-concave random spaces $X_{B_m}$ have extremal Banach--Mazur distance from the class $\mathcal{U}_n$ of $n$-dimensional normed spaces with a $1$-unconditional basis when $m\approx n$. 
The precise statement is the following.

\begin{theorem}\label{th:1.3}
Let $\mu$ be an isotropic log-concave probability measure on $\mathbb{R}^n$. 
If $C\ln (\delta^{-1} )n\ls m\ls\exp(\sqrt{n})$, where $C>0$ is an absolute constant, then a random space $X_{B_m}$ of the pure random polytope model 
satisfies, with probability greater than $1-\delta$,
$$d_{{\rm BM}}(X_{B_m},Y)\gr \frac{c\sqrt{n}}{\sqrt{\ln \left (1+m/n\right )}}$$
for every $n$-dimensional normed space $Y$ with a $1$-unconditional basis.
\end{theorem}

\medskip

\noindent\textbf{Comparison with previous work.} 
Gluskin's original construction \cite{Gluskin-1981} established the extremal order $n$ of the Banach--Mazur distance between random Gaussian polytopes, providing a sharp estimate for the diameter of the Banach--Mazur compactum. 
Subsequent work focused primarily on refinements within the Gaussian or highly symmetric settings. 
The first named author and Hartzoulaki initiated this line of research in \cite{Giannopoulos-Hartzoulaki-2002}; they considered the case where the $x_j$'s are independent random vertices of the cube, determined the asymptotic shape of a random $B_m$ and used this information to deduce various properties of the corresponding random space $X_{B_m}$ (see also the work of Mendelson, Pajor and Rudelson \cite{Mendelson-Pajor-Rudelson-2005} where it is proved that the combinatorial dimension, entropy and Gelfand numbers of these polytopes are extremal at every scale of their arguments).

These results were extended by Litvak, Pajor, Rudelson and Tomczak-Jaegermann in \cite{Litvak-Pajor-Rudelson-Tomczak-2005} to random spaces $X_{B_m}$ generated by a random vector $X=(\xi_1,\ldots,\xi_n)$ whose coordinates are independent copies of a random variable $\xi$ with expectation $\mathbb{E}(\xi)=0$ and variance ${\rm Var}(\xi)=1$, satisfying $\left(\mathbb{E}|\xi|^p\right)^{1/p}\ls L\sqrt{p}$ for some constant $L>0$ and every $p\gr 1$ (we then say that $\xi$ is $L$-subgaussian). 
Subsequently, Lata{\l}a, Mankiewicz, Oleszkiewicz and Tomczak-Jaegermann showed in \cite{LMOTJ-2007} that if $m\approx n$, then for a random pair of such spaces 
the Banach--Mazur distance is of order $n$ and they obtained a lower bound for the basis constant for $m\gr n^2(\ln n)$.

\smallskip 

Our results extend this phenomenon to the much broader class of isotropic log-concave measures. 
We exploit several of the ideas that appear in \cite{LMOTJ-2007}. 
However, unlike the Gaussian or subgaussian settings, general log-concave measures lack rotational invariance and do not admit a product structure. 
As a result, many classical tools used in the analysis of random polytopes are no longer directly applicable. 

To overcome these difficulties, we combine techniques from geometric functional analysis with probabilistic tools for log-concave measures, together with volumetric and operator-theoretic arguments. 
A crucial role is played by recent advances related to Bourgain's slicing problem, which provide uniform control of isotropic constants and enable quantitative volumetric estimates.

\medskip 

\noindent \textbf{Organization of the paper.} 
Section~\ref{section:2} contains background information and auxiliary entropy estimates. 
In Section~\ref{section:3} we collect those tools from the theory of isotropic log-concave probability measures that play a key role in our study.   
Section~\ref{section:4} presents the proof of the main Banach--Mazur distance estimate for both random models that we consider. 
Section~\ref{section:5} is devoted to the estimate of the distance from the class of spaces with a $1$-unconditional basis. 
Finally, Section~\ref{section:6} develops the operator-theoretic framework of mixing operators and derives bounds for norms of projections and basis constants of isotropic log-concave random spaces. 

\section{Background information and auxiliary results}\label{section:2}

\noindent \textbf{\S~2.1. Notation and definitions.} 
We work in $\mathbb{R}^n$, equipped with the standard inner product $\langle \cdot, \cdot \rangle$. 
The associated Euclidean norm is denoted by $|\cdot|$, the Euclidean unit ball by $B_2^n$ and the Euclidean unit sphere by $S^{n-1}$.
We also fix an orthonormal basis $\{e_1,\ldots ,e_n\}$. Lebesgue measure in $\mathbb{R}^n$ is denoted by $\vol_n$ and we write 
$\omega_n = \vol_n(B_2^n)$ for the volume of the Euclidean unit ball.
We write $M_n(\mathbb{R})$ for the set of linear operators $T:\mathbb{R}^n\to\mathbb{R}^n$. If $T\in M_n(\mathbb{R})$, then we write $T\in GL_n$ if $T$ is invertible and $T\in SL_n$ if $T$ is volume preserving. 
We also write 
$$\|x\|_p=\left(\sum_{i=1}^n|x_i|^p\right)^{1/p}$$
for the $\ell_p^n$-norm of $x=(x_1,\ldots,x_n)\in {\mathbb R}^n$, $1\ls p<\infty$ (in the case $p=\infty$, $\|x\|_{\infty}=\max_{i\ls n}|x_i|$).
We use the notation $|N|$ for the cardinality of a finite set $N$. 
For any $m\gr 1$, we denote $[m]:=\{1,\ldots,m\}$.

The letters $c,c_1,c_2,c^{\prime}$, etc., are reserved for absolute positive constants, which may change from line to line. 
Wherever we write $a\approx b$, this means that there exist absolute constants $c_1,c_2>0$ such that $c_1a\ls b\ls c_2a$. 

A convex body in $\mathbb{R}^n$ is a compact convex set $K$ with nonempty interior.  
It is called centrally symmetric if $K = -K$. For any finite set $A\subset\mathbb{R}^n$, the convex hull ${\rm conv}(A)$ of $A$ is the set
of all convex combinations of points from $A$. The absolute convex hull of $A$ is the centrally symmetric convex body ${\rm absconv}(A\cup (-A))$.  
The support function of a convex body $K$ is the function $h_K(y)=\max\{\langle x,y\rangle :x\in K\}$ for $y\in\mathbb{R}^n$. 
The mean width of $K$ is defined by
$$w(K)=\int_{S^{n-1}}h_K(\xi)\, d\sigma(\xi),$$
where $\sigma$ denotes the rotationally invariant probability measure on $S^{n-1}$.

All $d$-dimensional normed spaces in this work are of the form $X=({\mathbb R}^d,\|\cdot\|)$. 
The unit ball of $X$ is a centrally symmetric convex body in ${\mathbb R}^d$, denoted by $B_X$. 
Conversely, every centrally symmetric convex body $K$ induces the norm $\|x\|_K=\min\{\lambda\gr 0:x\in \lambda K\}$ on ${\mathbb R}^d$ 
and $K$ is the unit ball of $X_K=({\mathbb R}^d,\|\cdot\|_K)$. 
The dual norm is defined by $\|y\|_{\ast}=\max\{|\langle x,y\rangle|:x\in B_X\}$ and the unit ball of $X^{\ast}=({\mathbb R}^n,\|\cdot\|_{\ast})$ 
is the polar body $B_{X^{\ast}}=B_X^{\circ}$ of $B_X$.

Let $X$ and $Y$ be two $n$-dimensional normed spaces. 
Their Banach--Mazur distance $d_{{\rm BM}}(X,Y)$ is defined by
$$d_{{\rm BM}}(X,Y)=\min\{ \|T:X\to Y\|\;\|T^{-1}:Y\to X\| \mid T:X\to Y \ \text{is a linear isomorphism}\}.$$
John's theorem \cite{John-1948} shows that $d_{{\rm BM}}(X,\ell_2^n)\ls\sqrt{n}$ for every $X$. 
It follows that $d_{{\rm BM}}(X,Y)$ is always bounded by $n$. 
On the other hand, as we already mentioned, Gluskin \cite{Gluskin-1981} proved that there exists an absolute constant $c>0$ such that for every $n$ 
one can find $n$-dimensional normed spaces $X_n,Y_n$ with $d_{{\rm BM}}(X_n,Y_n)\gr cn$.

For any $n$-dimensional normed space $X$ and any basis $\left\{x_{i}\right\}_{i=1}^{n}$ of $X$, the basis constant of $\left\{x_{i}\right\}_{i=1}^{n}$ is the smallest $K>0$ with the property that for every $1\ls k\ls n$ and any choice of scalars $\{\lambda_i\}_{i=1}^n$ one has
$$\left\|\sum_{i=1}^k\lambda_ix_i\right\|\ls K\,\left\|\sum_{i=1}^n\lambda_ix_i\right\|.$$
The basis constant of $X$, denoted by ${\rm bc}(X)$, is defined as the infimum of the basis constants over all bases $\left\{x_{i}\right\}_{i=1}^{n}$ of $X$. 
By John's theorem it is clear that ${\rm bc}(X) \ls \sqrt{\dim X}$. 
This upper estimate is asymptotically sharp as shown by Szarek in \cite{Szarek-1983}. 

Let $n,k\in \mathbb{N}$, with $1\ls k\ls n/2$ and let $\beta>0$. 
Recall from the introduction that an operator $T \in M_n(\mathbb{R})$ is called $(k,\beta)$-mixing if there exists a subspace $E$ of $\mathbb{R}^{n}$, with ${\rm dim}(E)\gr k$, such that
$${\rm dist}(T(x), E)=|P_{E^{\perp}} T(x)|\gr\beta |x| \quad \text { for all } x \in E.$$
The set of all $(k, \beta)$-mixing operators on $\mathbb{R}^{n}$ is denoted by ${\rm Mix}_{n}(k,\beta)$.
The connection between mixing operators and the basis constant is given by the observation that any projection $P \in M_n(\mathbb{R})$ of rank $k\ls n/2$ is $(k,1/2)$-mixing (see e.g.~\cite{Mankiewicz-Tomczak-handbook}). 
On the other hand, by the definition of the basis constant, in any $n$-dimensional normed space $X$ there exists a projection $P$ of rank $k=\lfloor n/2\rfloor $ with $\|P:X\to X\|\ls {\rm bc}(X)$. 
Therefore, 
\begin{equation}\label{eq:low-bc}{\rm bc}(X)\gr c\inf\left\{\|T:X\to X\|:\;T\;\text{is}\;(\lfloor n/2\rfloor, 1/2)-\text{mixing}\right\}.\end{equation}

We refer the reader to the books \cite{AGA-book}, \cite{AGA-book-2}, \cite{Pisier-book} and \cite{Tomczak-book} for basic facts used throughout the text.

\bigskip 

\noindent \textbf{\S 2.2. Operator norms and volume estimates.} 
For any $T\in M_n(\mathbb R)$ we write $\|T\|_{{\rm HS}}$ for the Hilbert--Schmidt norm of $T$; recall that
$$\|T\|_{{\rm HS}}=\left(\sum_{i,j=1}^n|\langle T(e_i),e_j\rangle|^2\right)^{1/2}.$$
Then the space $M_n(\mathbb R)$ equipped with this norm can be identified with the Euclidean space ${\mathbb R}^{n^2}$.
We also define $\|T\|_{{\rm op}}=\|T:\ell_2^n\to\ell_2^n\|$ and
$$V_{{\rm op}}^n=\{T\in M_n(\mathbb R): \|T\|_{{\rm op}}\ls 1\}.$$
For any centrally symmetric convex body $B$ in $\mathbb{R}^n$ we define
$$V_B^n=\{T\in M_n(\mathbb R): \|T:\ell_1^n\to X_B\|\ls 1\}.$$
Note that $\|T:\ell_1^n\to X_B\|\ls 1$ if and only if $T(e_i)\in B$ for all $i=1,\ldots,n$. 
Therefore,
\begin{equation}\label{eq:volume-VB}
\vol_{n^2}(V_B^n)=\vol_n(B)^n.
\end{equation}

The next proposition provides a lower bound for the volume of the ``unit ball'' on $M_n(\mathbb R)$ that corresponds to the operator norm.
For a proof, see \cite[Lemma~7]{Mankiewicz-Tomczak-handbook} or \cite[Proposition~6.1.8]{AGA-book-2}.

\begin{proposition}\label{prop:volume-V}
For every $n\gr 1$ and every centrally symmetric convex body $B\subset \mathbb{R}^n$ we have
$$\vol_{n^2}(V_{{\rm op}}^n)\gr \left(\frac{c_2}{\sqrt{n}}\right)^{n^2},$$
where $c_2>0$ is an absolute constant.
\end{proposition}

\medskip 

\noindent \textbf{\S~2.3. Covering estimates.} 
Let $K$ and $L$ be convex bodies in $\mathbb{R}^{n}$. 
The covering number $N(L,K)$ is the smallest number of translates of $K$ needed to cover $L$:
$$N(L,K)= \min\Bigl\{ N\in\mathbb{N} : \exists\, x_1, \ldots, x_N \in \mathbb{R}^n \text{ such that } L \subseteq \bigcup_{j=1}^{N} (x_{j}+K)\Bigr\}.$$
Assume that $K$ is centrally symmetric and let $D\subseteq L$. 
We say that a finite set $\mathcal{N}\subset D$ is a $t$-net of $D$ with respect to the norm $\|\cdot\|_K$ induced by $K$ if for every $x\in D$ there exists $y\in\mathcal{N}$ such that $\|x-y\|_K\ls t$. 
It is easily checked that, for $0<t\ls 1$, if $\mathcal{N}$ is a $t$-net of $D$ with respect to $\|\cdot\|_K$ then $N(D,K)\ls |\mathcal{N}|$.

\medskip 

We shall use the following standard volumetric lemma.

\begin{lemma}\label{lem:standard-net}
Let $K$ and $L$ be two centrally symmetric convex bodies in $\mathbb R^d$ such that $K\subseteq L$ and let $\mathcal A\subseteq L$. 
Then there exists a $1$-net $\mathcal N\subseteq \mathcal A$ with respect to the norm $\|\cdot\|_K$ such that
$$ |\mathcal N| \ls 3^d\,\frac{\vol_d(L)}{\vol_d(K)}. $$
\end{lemma}

\begin{proof}
Let $\mathcal N=\{x_1,\ldots,x_N\}$ be a maximal subset of $\mathcal A$ with respect to inclusion such that $\|x_i-x_j\|_K\gr 1$ for all $i\ne j.$
Then $\mathcal N$ is a $1$-net of $\mathcal A$. Also, $\bigcup_{i=1}^N\left(x_i+\frac12K\right)\subseteq L+\frac12K\subseteq \frac32L,$ and the sets $x_i+\frac12K$ have pairwise disjoint interiors. 
Indeed, if $\operatorname{int}\left(x_i+\frac12K\right)\cap \operatorname{int}\left(x_j+\frac12K\right)\neq\emptyset,$ then $x_i-x_j\in \operatorname{int}(K)$, hence $\|x_i-x_j\|_K<1$, a contradiction.
Hence $N\,\vol_d\!\left(\frac12K\right)\ls \vol_d\!\left(\frac32L\right),$ which gives $|\mathcal N|=N \ls 3^d\,{\vol_d(L)}/{\vol_d(K)}.$
\end{proof}

From Lemma~\ref{lem:standard-net} and Proposition~\ref{prop:volume-V} we deduce the following entropy estimate, which is a Euclidean-ball specialization of \cite[Proposition~5.3]{MT1}.

\begin{proposition}\label{prop:entropy-general}
Let $x_1,\dots,x_n$ be linearly independent vectors in $\mathbb R^n$ and set $ \delta_0:=\big(\max_{1\ls j\ls n}|x_j|\big)^{-1}.$
Let $B\subset \mathbb R^n$ be a centrally symmetric convex body such that $rB_2^n\subseteq B$ for some $r>0$. 
Define $\mathcal A:=\{R\in M_n(\mathbb R):R(x_i)\in B,\ i=1,\dots,n\}.$
Then for every $0<t\ls r\delta_0$ and every subset $\mathcal A_0\subseteq \mathcal A$, there exists a $t$-net $\mathcal N_t\subseteq \mathcal A_0$ with respect to the Euclidean operator norm such that
$$|\mathcal N_t| \ls \left(\frac{C_0}{t}\right)^{n^2} \left( \frac{\vol_n(B)\,\vol_n(B_1^n)} {\vol_n(B_2^n)\,\vol_n(\operatorname{absconv}\{x_1,\dots,x_n\})} \right)^n, $$
where $C_0>0$ is an absolute constant.
\end{proposition}

\begin{proof}
Let $T\in M_n(\mathbb R)$ be defined by $T(e_i)=x_i,\, i=1,\dots,n. $
Then $ T(B_1^n)=\operatorname{absconv}\{x_1,\dots,x_n\},$ and therefore $ |\det T| = {\vol_n(\operatorname{absconv}\{x_1,\dots,x_n\})}/{\vol_n(B_1^n)}. $

Consider the linear map $\Phi_T:M_n(\mathbb R)\to M_n(\mathbb R),\, \Phi_T(R)=RT.$ By definition of $\mathcal A$ we have 
$\Phi_T(\mathcal A) = \{S\in M_n(\mathbb R):S(e_i)\in B,\ i=1,\dots,n\} = V_B^n.$ If we identify $M_n(\mathbb R)$ with $\mathbb R^{n^2}$ through 
the rows of a matrix, then $\Phi_T$ acts on each row by right multiplication by $T$. 
Hence the matrix of $\Phi_T$ is block diagonal with $n$ diagonal blocks equal to $T$ and therefore $|\det \Phi_T|=|\det T|^n.$
It follows that $\vol_{n^2}(\Phi_T(\mathcal A))=|\det T|^n\,\vol_{n^2}(\mathcal A).$
Using \eqref{eq:volume-VB}, this gives
\begin{equation}\label{eq:identity-Phi}\vol_{n^2}(\mathcal A) = \left( \frac{\vol_n(B)\,\vol_n(B_1^n)} {\vol_n(\operatorname{absconv}\{x_1,\dots,x_n\})} \right)^n.\end{equation}
Next, we claim that $ tV_{\rm op}^n\subseteq \mathcal A.$
Indeed, if $S\in tV_{\rm op}^n$, then $\|S\|_{\rm op}\ls t$ and therefore for every $i=1,\dots,n$, $|S(x_i)|\ls \|S\|_{\rm op}|x_i|\ls t\max_j|x_j|\ls r. $
Thus $S(x_i)\in rB_2^n\subseteq B$, which proves the claim.

We now apply Lemma~\ref{lem:standard-net} in the $n^2$-dimensional space $M_n(\mathbb R)\simeq \mathbb R^{n^2}$ with $ K=tV_{\rm op}^n,\, L=\mathcal A,\, D=\mathcal A_0.$
Since the norm induced by $tV_{\rm op}^n$ is precisely $t^{-1}\|\cdot\|_{\rm op}$, we obtain a $t$-net $\mathcal N_t\subseteq \mathcal A_0$ with respect to $\|\cdot\|_{{\rm op}}$ 
such that
$$ |\mathcal N_t| \ls 3^{n^2}\frac{\vol_{n^2}(\mathcal A)}{\vol_{n^2}(tV_{\rm op}^n)} = 3^{n^2}t^{-n^2}\frac{\vol_{n^2}(\mathcal A)}{\vol_{n^2}(V_{\rm op}^n)}
=\left(\frac{3}{t}\right)^{n^2} \left( \frac{\vol_n(B)\,\vol_n(B_1^n)} {\vol_n(\operatorname{absconv}\{x_1,\dots,x_n\})} \right)^n \frac{1}{\vol_{n^2}(V_{\rm op}^n)},  $$
where in the last equality we substituted $\vol_{n^2}(\mathcal A)$, using the identity~\eqref{eq:identity-Phi}.

Finally, using Proposition~\ref{prop:volume-V} and the fact that $\vol_n(B_2^n)^{1/n}\approx 1/\sqrt{n}$, we get that  
$\vol_{n^2}(V_{\rm op}^n)^{-1} \ls {C^{n^2}}/{\vol_n(B_2^n)^n}.$ This completes the proof.
\end{proof}

\section{Isotropic log-concave measures}\label{section:3}

We say that a Borel probability measure $\mu$ on $\mathbb R^n$ is log-concave 
if $\mu(H)<1$ for every hyperplane $H$ in ${\mathbb R}^n$ (in which case we say that $\mu$ is full-dimensional) and $\mu(\lambda A+(1-\lambda)B) \gr \mu(A)^{\lambda}\mu(B)^{1-\lambda}$ for any pair of compact sets $A,B$ in ${\mathbb R}^n$ and any $\lambda \in (0,1)$. 
Borell \cite{Borell-1974} proved that, under these assumptions, $\mu$ has a log-concave density $f$. 

Let $f:\mathbb{R}^n\to [0,\infty)$ be a log-concave function with finite, positive integral. 
Its barycenter is defined by
$$\operatorname{bar}(f)= \frac{\int_{\mathbb{R}^n} x\, f(x)\, dx}{\int_{\mathbb{R}^n} f(x)\, dx}.$$
We say that $f$ is centered if $\operatorname{bar}(f)=0$. 
The isotropic constant of a log-concave function $f$ with finite positive integral is the affine-invariant quantity
\begin{equation}\label{eq:definition-isotropic}
L_f:= \left( \frac{\|f\|_{\infty}}{\int_{\mathbb{R}^n} f(x)\, dx} \right)^{1/n} \det(\operatorname{Cov}(f))^{1/(2n)},
\end{equation}
where $\operatorname{Cov}(f)$ denotes the covariance matrix of $f$. 
A log-concave function $f$ is called isotropic if
$$\operatorname{bar}(f)=0,\quad \int_{\mathbb{R}^n}f(x)\,dx=1,\quad \text{and} \quad \operatorname{Cov}(f)=I_n$$
where $I_n$ is the identity operator on $\mathbb{R}^n$. In this case, $L_f=\|f\|_{\infty}^{1/n}$. 
A full-dimensional log-concave probability measure $\mu$ on $\mathbb{R}^n$ is called isotropic if its density $f$ is isotropic. 
We then set $L_{\mu}:=L_f$.

Let $\mu$ and $\nu$ be two log-concave probability measures on $\mathbb{R}^n$. Let $T:\mathbb{R}^n\to \mathbb{R}^n$ be a measurable function which is defined $\mu$-almost everywhere and satisfies
$$\nu (B)=\mu(T^{-1}(B))$$
for every Borel subset $B$ of $\mathbb{R}^n$. 
We then say that $T$ pushes forward $\mu$ to $\nu $ and write $T_\ast\mu=\nu$.
It is not hard to check that for every centered log-concave probability measure $\mu$ on $\mathbb{R}^n$ there exists $T\in GL_n$ such that the log-concave probability measure $T_{\ast }\mu$ is isotropic, and $L_{T_{\ast}\mu}=L_{\mu}$ (see \cite[Section~2.3]{BGVV-book}). 
In other words, every centered log-concave probability measure $\mu$ admits an isotropic position. 

It is also known (see \cite[Proposition~2.3.12]{BGVV-book}) that $L_{\mu}\gr c$ for every isotropic log-concave measure $\mu$ on $\mathbb{R}^n$, where $c>0$ is an absolute constant. 
An equivalent formulation of Bourgain's slicing problem \cite{Bourgain-1986} asks if there exists an absolute constant $C>0$ such that
\begin{equation}\label{eq:conjecture}
L_n:=\max\{ L_{\mu}:\mu \text{ is an isotropic log-concave measure on } \mathbb{R}^n\}\ls C
\end{equation}
for all $n\gr 1$. 
An affirmative solution was recently obtained by Klartag and Lehec~\cite{Klartag-Lehec-2025}, building on an important contribution by Guan~\cite{Guan-preprint} (see also~\cite{Bizeul-2025} for an alternative proof). 
Consequently, $L_{\mu}\approx 1$, uniformly in $n$, for every isotropic log-concave measure $\mu$ on $\mathbb{R}^n$.
For further background, we refer to the survey~\cite{Giannopoulos-Pafis-Tziotziou-2025}.

Let $\mu$ be a full-dimensional log-concave probability measure on $\mathbb{R}^n$.
For any $1\ls k \ls n-1$ and any $k$-dimensional subspace $F$ of $\mathbb{R}^n$, the marginal of $\mu$ onto $F$ is defined by
$$\pi_F(\mu)(B):=\mu(P_F^{-1}(B)),$$
for every Borel set $B\subset F$. 
The measure $\pi_F(\mu)$ is log-concave and admits a density
$$(\pi_F f)(x)=\int_{F^\perp} f(y+x)\,dy.$$
If $f$ is centered (respectively isotropic), then so is $\pi_F f$ (see~\cite[Proposition~5.1.11]{BGVV-book}). 

A very useful deviation inequality of Paouris \cite{Paouris-2006} asserts that if $\mu$ is an isotropic log-concave measure  on ${\mathbb R}^n$, then
\begin{equation}\label{eq:paouris-1}
\mu \left(\{x\in {\mathbb R}^n:\ |x|\gr  ct\sqrt{n}\}\right)\ls \exp(-t\sqrt{n})
\end{equation}
for every $t\gr  1$, where $c>0$ is an absolute constant. 
We shall also use a small ball probability estimate of Paouris \cite{Paouris-2012} which is now available in its optimal form, due to Bizeul \cite{Bizeul-2025}: if $\mu$ is an isotropic log-concave measure on $\mathbb R^n$ then, for any $0<\varepsilon\ls c_0$ and any $y\in\mathbb{R}^n$,
\begin{equation}\label{eq:small-ball}
\mu\left(\{x\in\mathbb{R}^n:|x-y|^2 \ls \varepsilon n\}\right) \ls  \varepsilon^{c_0n}
\end{equation}
where $c_0>0$ is an absolute constant. 
These two facts will be used in the following form: there exist absolute constants $\varepsilon_0,b,c>0$ such that, if $m\ls \exp(\sqrt{n})$ then $m$ independent random vectors $x_1,\ldots ,x_m$ distributed according to $\mu$ satisfy
\begin{equation}\label{eq:euclidean-norm}
\varepsilon_0\sqrt{n}\ls |x_i|\ls b\sqrt{n},\qquad i=1,\ldots ,m
\end{equation}
with probability greater than $1-\exp(-c\sqrt{n})$.

Let $\mu$ be a centered log-concave probability measure on $\mathbb R^n$ with density $f$. 
For any $p\gr 1$ we define the $L_p$-centroid body $Z_p(\mu)$ of $\mu$ as the convex body whose support function is
$$ h_{Z_p(\mu)}(y):=\left(\int_{\mathbb R^n} |\langle x,y\rangle|^p f(x)\,dx \right)^{1/p},\qquad y\in \mathbb{R}^n.$$ 
The convex bodies $Z_p(\mu)$ are always centrally symmetric, and $Z_p(T_{\ast}\mu)=T(Z_p(\mu))$ for every $T\in GL_n$ and $p\gr 1$.
Note that if $\mu$ is isotropic then $Z_2(\mu)=B_2^n$. 
Also, classical reverse H\"{o}lder inequalities (see \cite[Chapter~3]{BGVV-book}) show that 
\begin{equation}\label{eq:zq-mu-inclusions}
Z_p(\mu)\subseteq Z_q(\mu)\subseteq \frac{cq}{p}Z_p(\mu)
\end{equation}
for all $1\ls p<q$, where $c>0$ is an absolute constant. 
If $\mu$ is isotropic, then combining previous results of Lutwak, Yang and Zhang \cite{Lutwak-Yang-Zhang-2000} and Paouris \cite{Paouris-2006} with the affirmative answer to Bourgain's problem, we know that
\begin{equation}\label{eq:Zp-volume} 
\vol_n(Z_p(\mu))^{1/n}\approx \sqrt{p/n}
\end{equation}
for every $1\ls p\ls n$ (see \cite[Chapter~5]{BGVV-book} for a thorough discussion).

A general study of the asymptotic shape of random polytopes whose vertices have a log-concave distribution was initiated by Dafnis, Giannopoulos and Tsolomitis in \cite{Dafnis-Giannopoulos-Tsolomitis-2009}. 
They showed that for any isotropic log-concave measure $\mu$ on ${\mathbb R}^n$ and any $cn\ls m \ls e^n$, the random polytope $B_m$ defined by $m$ independent random points $x_1,\ldots ,x_m$ distributed according to $\mu$ satisfies the inclusion
\begin{equation}\label{eq:DGT}
B_m\supseteq c_1Z_{\ln (1+m/n)}(\mu)
\end{equation}
with probability greater than $1-\exp\left (-c_0m\right)$, where $c_0>0$ is an absolute constant (see \cite[Chapter~11]{BGVV-book} for a thorough discussion). 
Combining this fact with \eqref{eq:Zp-volume} we see that
\begin{equation}\label{eq:Bm-volume-lower}
\vol_n(B_m)^{1/n}\gr \frac{c_1\sqrt {\ln (1+m/n)}}{\sqrt{n}}
\end{equation}
with probability greater than $1-\exp\left (-c_0m\right)$. 
On the other hand, it is also proved in \cite{Dafnis-Giannopoulos-Tsolomitis-2009} that for every $n\ls m\ls e^n$ one has 
\begin{equation}\label{eq:Bm-volume-upper}
\vol_n(B_m)^{1/n}\ls \frac{c_2\sqrt {\ln (1+m/n)}}{\sqrt{n}}
\end{equation}
with probability greater than $1-1/m$. In the restricted range $n\ls m\ls \exp(\sqrt{n})$, the same upper bound holds with probability
at least $1-\exp(-c\sqrt{n})$. By the Blaschke--Santal\'{o} and Bourgain--Milman inequalities (see e.g.~\cite[Chapter~8]{AGA-book}) we see that
\begin{equation}\label{eq:Bm-polar-volume}
\vol_n(B_m^{\circ })^{1/n}\approx \frac{1}{\sqrt {n\ln (1+m/n)}}\end{equation}
for a random $B_m$.

\smallskip 

We refer to~\cite{AGA-book} and \cite{AGA-book-2} for asymptotic convex geometry and to~\cite{BGVV-book} for background on isotropic convex bodies and log-concave measures.

\section{Gluskin's theorem for log-concave random spaces}\label{section:4}

In this section we present the proof of Theorem~\ref{th:1.1}. 
We provide two arguments, corresponding to the two different models of Gluskin random polytopes that were described in the introduction.

\bigskip 

\noindent \textbf{\S~4.1. Basis-enriched model.} Our first goal is to study the centrally symmetric random polytope
$$ B_m:=B_m(\omega)={\rm absconv}\left\{e_1,\ldots,e_n, \frac{x_1(\omega)}{\sqrt{n}},\ldots,\frac{x_m(\omega)}{\sqrt{n}}\right\}$$
where $\{e_1,\ldots ,e_n\}$ is the standard orthonormal basis of $\mathbb{R}^n$ and $x_1,\ldots,x_m$ are independent random vectors defined on 
some probability space $\Omega$, each distributed according to an isotropic log-concave probability measure $\mu$ on $\mathbb{R}^n$. 
We denote by $X_{B_m}$ the normed space whose unit ball is $B_m$. We also consider an independent copy of this construction on some probability space $\widetilde{\Omega}$ and denote the corresponding random polytope by $A_m=A_m(\tilde{\omega})$.
We write $Y_{A_m}$ for the normed space whose unit ball is $A_m$.

We prove the following theorem.

\begin{theorem}[log-concave Gluskin's theorem]\label{th:log-concave-gluskin-1}
Let $2n\ls m\ls\exp(\sqrt{n})$ and let  $x_1,\ldots,x_m$ and $y_1,\ldots ,y_m$ be independent random vectors in $\R^n$ with isotropic log-concave distribution $\mu$. 
Then, the random polytopes 
$$B_m={\rm absconv}\left\{e_1,\ldots,e_n, \frac{x_1}{\sqrt{n}},\ldots,\frac{x_m}{\sqrt{n}}\right\}\quad\text{and}
\quad A_m={\rm absconv}\left\{e_1,\ldots,e_n, \frac{y_1}{\sqrt{n}},\ldots,\frac{y_m}{\sqrt{n}}\right\}$$ 
satisfy
$$d_{{\rm BM}}(Y_{A_m},X_{B_m})\gr cn/\ln (1+m/n)$$
where $c>0$ is an absolute constant, with probability tending to $1$ as $n\to\infty$, uniformly in $m$.
\end{theorem}

Clearly, choosing $m\approx n$ we recover Gluskin's sharp lower bound $d_{{\rm BM}}(Y_{A_m},X_{B_m})\gr cn$.

\bigskip 

We start with a few basic geometric properties of the random polytope $B_m$ that will be needed for the distance estimate.
Since $B_m\supseteq {\rm absconv}\{e_1,\ldots,e_n\}=B_1^n$, we clearly have $\sqrt n\,B_m\supseteq B_2^n$. 
On the other hand, \eqref{eq:euclidean-norm} shows that if $m\ls \exp(\sqrt{n})$ then $|x_j|\ls b\sqrt{n}$ for all $j=1,\ldots ,m$ with probability greater than $1-\exp(-c\sqrt{n})$.
Possibly increasing $b$, we may assume that $b\gr 1$. 
We summarize these facts in the next proposition.

\begin{proposition}[geometry of $B_m$]\label{prop:geometry}
Let $n\ls m\ls \exp(\sqrt{n})$. If $x_1,\ldots,x_m$ are independent random vectors in $\R^n$ with isotropic log-concave distribution $\mu$, then
$$B_2^n\subseteq \sqrt{n}B_m\subseteq b\sqrt{n}\,B_2^n $$
with probability greater than $1-\exp(-c\sqrt{n})$, where $b\gr 1$ and $c>0$ are absolute constants.
\end{proposition}

In what follows, we set
\begin{equation}\label{eq:omega}
\Omega_0 = \left\{\omega\in\Omega:\ B_2^n\subseteq \sqrt{n}\,B_m(\omega)\subseteq b\sqrt{n}\,B_2^n\right\}\quad\text{and}\quad \widetilde{\Omega}_0 = \left\{\tilde{\omega}\in \widetilde{\Omega}:\ B_2^n\subseteq \sqrt{n}\,A_m(\tilde{\omega})\subseteq b\sqrt{n}\,B_2^n\right\}.
\end{equation}
From Proposition~\ref{prop:geometry} we know that if $n\ls m\ls \exp(\sqrt{n})$, then 
\begin{equation}\label{eq:measure-omega}
\mathbb{P}(\Omega_0)=\mathbb{P}(\widetilde{\Omega}_0)\gr  1-\exp(-c \sqrt{n}).
\end{equation}

We start with an elementary lemma.

\begin{lemma}\label{lem:measure-vs-volume}
Let $\nu$ be an isotropic log-concave probability measure on $\R^k$ and let $f_\nu$ denote its density. 
Then, for every convex body $K$ in ${\mathbb R}^k$ and every $ \alpha>0$ one has
$$ \nu(\alpha K)\ls \|f_{\nu}\|_{\infty}\alpha^k\vol_k(K)\ls (C_1\alpha)^k\vol_k(K),$$
where $C_1>0$ is an absolute constant.
\end{lemma}

\begin{proof}
We simply write
$$\nu(\alpha K)=\int_{\alpha K}f_{\nu}(x)\,dx\ls \|f_{\nu}\|_{\infty}\vol_k(\alpha K)=\|f_{\nu}\|_{\infty}\alpha^k\vol_k(K).$$
Since $\nu$ is isotropic, its isotropic constant satisfies $L_\nu=\|f_\nu\|_\infty^{1/k}$ and hence $\|f_\nu\|_\infty=L_\nu^k\ls C_1^k,$ where $C_1>0$ is an absolute constant. 
\end{proof}

The next volume estimate is also standard (see \cite{Barany-Furedi-1987} and \cite{Gluskin-1989}).

\begin{lemma}\label{lem:barany-furedi} 
Let $m\gr k$ and let $Q={\rm absconv}\{z_1,\ldots ,z_m\}\subset rB_2^k$ for some $r>0$. 
Then,
$$ \vol_k(Q)\ls \left(\frac{C_2r\sqrt{\ln(1+m/k)}}{\sqrt{k}}\right)^k \vol_k(B_2^k),$$
where $C_2>0$ is an absolute constant.
\end{lemma}

An immediate consequence of the above is the following:

\begin{proposition}\label{prop:gluskin-lower-bound-4}
Let $T\in SL_n$ and let $\omega\in \Omega_0$. Then, for any $ \delta>0$ we have
$$ \mathbb{P}\left(\left\{ \tilde{\omega}\in \widetilde{\Omega}_0:\|T: Y_{A_m(\tilde{\omega})}\to X_{B_m(\omega)}\| \ls \delta\sqrt{n}\right\} \right) \ls\left(C_3b\delta L_\mu\sqrt{\ln(1+m/n)}\right)^{mn}.$$
\end{proposition}

\begin{proof}
Let $\delta>0$. 
If $\|T:Y_{A_m(\tilde{\omega})}\to X_{B_m(\omega)}\|\ls \delta\sqrt n$, then since $y_j(\tilde{\omega})/\sqrt n\in A_m(\tilde{\omega})$ for every $j=1,\ldots,m$, we have  $T(y_j(\tilde{\omega})/\sqrt n)\in \delta\sqrt n\,B_m(\omega)$ and hence $y_j(\tilde{\omega})\in \delta n\,T^{-1}(B_m(\omega))$ for all $j$. 
Therefore,
\begin{align*}
&\mathbb{P}\left(\left\{\tilde{\omega}\in \widetilde{\Omega}_0:\ \|T:Y_{A_m(\tilde{\omega})}\to X_{B_m(\omega)}\|\ls \delta\sqrt n\right\}\right)\\
&\qquad\ls \mathbb{P}\left(\left\{\tilde{\omega}\in \widetilde{\Omega}_0:\ y_j(\tilde{\omega})\in \delta n\,T^{-1}(B_m(\omega))\ \text{for all }j=1,\ldots,m\right\}\right) = \left(\mu(\delta n\,T^{-1}(B_m(\omega)))\right)^m,
\end{align*}
where we used independence of the random vectors $y_j(\tilde{\omega})$ in the last step.

By Lemma~\ref{lem:measure-vs-volume}, $ \mu(\delta n\,T^{-1}(B_m(\omega))) \ls (\delta nL_\mu)^n\,\vol_n(T^{-1}(B_m(\omega))).$
Since $T$ is volume preserving, we have $ \vol_n(T^{-1}(B_m(\omega)))=\vol_n(B_m(\omega)). $

Moreover, because $\omega\in \Omega_0$, we have $B_m(\omega)\subseteq bB_2^n$. 
Hence Lemma~\ref{lem:barany-furedi} gives
$$ \vol_n(B_m(\omega)) \ls \left(\frac{C_2b\sqrt{\ln(1+(m+n)/n)}}{\sqrt{n}}\right)^n\vol_n(B_2^n) \ls \left(\frac{C_3b\sqrt{\ln(1+m/n)}}{n}\right)^n,$$
since $\vol_n(B_2^n)^{1/n}\approx 1/\sqrt{n}$. 
Combining the previous estimates, we obtain
$$ \mathbb{P}\left(\left\{\tilde{\omega}\in \widetilde{\Omega}_0:\ \|T:Y_{A_m(\tilde{\omega})}\to X_{B_m(\omega)}\|\ls \delta\sqrt n\right\}\right)  \ls \left(C_3b\delta L_\mu\sqrt{\ln(1+m/n)}\right)^{mn}$$
as claimed.
\end{proof}

\begin{definition}
Let $b>0$ be the constant from Proposition~\ref{prop:geometry}. Recall that for any centrally symmetric convex body $B\subset \mathbb R^n$ we denote
$$V_B^n=\{T\in M_n(\mathbb{R}):\|T:\ell_1^n\to X_B\|\ls 1\}=\{T\in M_n(\mathbb R) : T(e_i)\in B, \; i=1,\ldots,n\}.$$
We define
\begin{equation}\label{eq:MB}\mathcal{M}_B=SL_n\cap \sqrt n\,V_B^n.\end{equation}
\end{definition}

With this notation we have the following entropy estimate.

\begin{proposition}\label{prop:entropy-first-model} 
Let $\omega\in \Omega_0$ and $0<t<1$. 
Then there exists a $t$-net $\mathcal{N}_t(\omega)$, with respect to the operator norm, for the set $\mathcal{M}_{B_m(\omega)}$ such that
$$ |\mathcal{N}_t(\omega)|\ls \left(\frac{C_4b\sqrt{\ln (1+m/n)}}{t}\right)^{n^2},$$
where $C_4>0$ is an absolute constant.
\end{proposition}

\begin{proof}
Set $B=\sqrt{n}B_m(\omega)$. 
Since $\omega\in\Omega_0$, we know that $B_2^n\subseteq B$. 
Moreover, 
$$\mathcal{M}_{B_m(\omega)}=\{T\in SL_n:T(e_i)\in B,\;i=1,\ldots ,n\}\subseteq V_B^n.$$
Thus, applying Proposition~\ref{prop:entropy-general} with $x_i=e_i$ and $r=1$, for every $0<t<1$ we may find a $t$-net $\mathcal{N}_t(\omega)$ for the set $\mathcal{M}_{B_m(\omega)}$ with
$$|\mathcal{N}_t(\omega)|\ls \left(\frac{C_0}{t}\right)^{n^2} \left(\frac{\vol_n(B)}{\vol_n(B_2^n)}\right)^n. $$
Since $B=\sqrt{n}B_m(\omega)\subseteq b\sqrt{n}B_2^n$, Lemma~\ref{lem:barany-furedi} gives
$$\vol_n(B)\ls \left(C_2b\sqrt{\ln((m+n)/n)}\right)^n\vol_n(B_2^n). $$
Combining the above, we obtain
$$|\mathcal{N}_t(\omega)|\ls \left(\frac{C_0C_2b\sqrt{\ln(1+m/n)}}{t}\right)^{n^2}. $$
This completes the proof with $C_4=C_0C_2$.
\end{proof}

\begin{proof}[\textbf{Proof of Theorem~$\mathbf{\ref{th:log-concave-gluskin-1}}$}] 
For each $\gamma>0$ and $m\in \mathbb{N}$ we define the following subsets of $\Omega_0\times \widetilde{\Omega}_0$:
$$ G(\gamma,m)=\{(\omega,\tilde{\omega}) : \;\hbox{there exists}\; T\in SL_n\;\hbox{ such that }\; \|T:Y_{A_m(\tilde{\omega})}\to X_{B_m(\omega)}\|\ls \gamma\sqrt{n}\}$$ 
and
$$ G^\ast(\gamma,m)=\{(\omega,\tilde{\omega}) : \;\hbox{there exists}\; T\in SL_n\;\hbox{ such that }\; \|T:X_{B_m(\omega)}\to Y_{A_m(\tilde{\omega})}\|\ls \gamma\sqrt{n}\}.$$
We also set
$$ H(\gamma,m)=(\Omega_0\times \widetilde{\Omega}_0)\setminus(G(\gamma,m)\cup G^\ast(\gamma,m)).$$
Let $\mathbb{P}$ also denote the product probability measure on $\Omega\times \tilde{\Omega}$. 
Then
$$ \mathbb{P}\left(H(\gamma,m)\right) = \mathbb{P}\left(\Omega_0\times \widetilde{\Omega}_0\right) - \mathbb{P}\left(G(\gamma,m)\cup G^\ast(\gamma,m)\right). $$
Therefore
$$ \mathbb{P}\left(H(\gamma,m)\right) \gr  \mathbb{P}\left(\Omega_0\times \widetilde{\Omega}_0\right) - \mathbb{P}\left(G(\gamma,m)\right) - \mathbb{P}\left(G^\ast(\gamma,m)\right). $$

By symmetry, $ \mathbb{P}\left(G^\ast(\gamma,m)\right)=\mathbb{P}\left(G(\gamma,m)\right). $
Moreover, by \eqref{eq:measure-omega}, $\mathbb{P}\left(\Omega_0\times \widetilde{\Omega}_0\right)\gr \left(1-e^{-c\sqrt n}\right)^2 \gr 1-2e^{-c\sqrt n}. $
Combining these estimates, we obtain
$$ \mathbb{P}\left(H(\gamma,m)\right) \gr  1-2\exp(-c\sqrt n)-2\mathbb{P}\left(G(\gamma,m)\right). $$

Thus, it is enough to prove that $\mathbb{P}(G(\gamma,m))\to 0$ uniformly for all integers $m$ satisfying $2n\ls m\ls \exp(\sqrt n).$
In fact, we shall prove the stronger estimate $\mathbb{P}(G(\gamma,m))\ls 2^{-nm}$ for a suitable choice of $\gamma$.
Combined with the previous estimate this implies that $\mathbb{P}(H(\gamma,m))\to 1$ as $n\to\infty$. 
Then, by the definition of $H(\gamma,m)$,  for every $(\omega,\tilde{\omega})\in H(\gamma,m)$ and every $T\in SL_n$ we have
$$\|T:Y_{A_m(\tilde{\omega})}\to X_{B_m(\omega)}\|\gr \gamma\sqrt{n}\quad\text{and}\quad \|T^{-1}:X_{B_m(\omega)}\to Y_{A_m(\tilde{\omega})}\|\gr \gamma\sqrt{n}.$$
Now, let $S\in GL_n$ and define $ T=|\det S|^{-1/n}S.$
Then $T\in SL_n$ and
\begin{align*}
&\|S:Y_{A_m(\tilde{\omega})}\to X_{B_m(\omega)}\|\cdot \|S^{-1}:X_{B_m(\omega)}\to Y_{A_m(\tilde{\omega})}\| \\
&\hspace*{1.5cm}= \|T:Y_{A_m(\tilde{\omega})}\to X_{B_m(\omega)}\|\cdot \|T^{-1}:X_{B_m(\omega)}\to Y_{A_m(\tilde{\omega})}\|\gr \gamma^2n.
\end{align*}
Therefore,
$$d_{{\rm BM}}(Y_{A_m(\tilde{\omega})},X_{B_m(\omega)})\gr \gamma^2n.$$

It remains to choose $\gamma$. 
We first fix $\omega\in \Omega_0$ and define the following subsets of $\widetilde{\Omega}_0$: 
$$ Q(\omega,\gamma,m)=\{\tilde{\omega}\in \widetilde{\Omega}_0 : (\omega,\tilde{\omega})\in G(\gamma,m)\}.$$
and
$$Q_t(\omega,\gamma,m)=\{\tilde{\omega}\in \widetilde{\Omega}_0 : \;\hbox{there exists}\; T\in \mathcal{N}_t(\omega)\;\hbox{with} \; \|T:Y_{A_m(\tilde{\omega})}\to X_{B_m(\omega)}\|\ls \gamma\sqrt{n}\},$$
where $0<t<1$ will be suitably chosen and $\mathcal{N}_t(\omega)$ is the $t$-net of $\mathcal{M}_{B_m(\omega)}$ from Proposition~\ref{prop:entropy-first-model}.

\begin{claim}
For any $\omega\in\Omega_0$ and $0<\gamma\ls 1$ one has $Q(\omega,\gamma,m)\subseteq Q_t\left(\omega,\gamma+bt,m\right).$
\end{claim}

\begin{proof}[Proof of the claim] 
Let $\tilde{\omega}\in Q(\omega,\gamma,m)$. Then, there exists $T\in SL_n$ such that 
$$\|T:Y_{A_m(\tilde{\omega})}\to X_{B_m(\omega)}\|\ls \gamma \sqrt{n}\ls \sqrt{n}.$$ 
Since $e_i\in A_m(\tilde{\omega})$ for every $i=1,\ldots,n$, we have $ T(e_i)\in \sqrt n\,B_m(\omega)$ and therefore $ T\in \sqrt n\,V^n_{B_m(\omega)}$. 
It follows that $T\in \mathcal{M}_{B_m(\omega)}$ and hence we may find $S\in \mathcal{N}_t(\omega)$ such that $ \|T-S:\ell_2^n\to \ell_2^n\|\ls t.$
Note that $\|I_n:Y_{A_m(\tilde{\omega})}\to \ell_2^n\|\ls b$ and $\|I_n:\ell_2^n\to X_{B_m(\omega)}\|\ls \sqrt{n}$ by Proposition~\ref{prop:geometry}.
Then, 
$$ \|T-S:Y_{A_m(\tilde{\omega})}\to X_{B_m(\omega)}\| \ls \|I_n:Y_{A_m(\tilde{\omega})}\to \ell_2^n\|\,\|T-S:\ell_2^n\to \ell_2^n\|\,\|I_n:\ell_2^n\to X_{B_m(\omega)}\| \ls bt\sqrt n.$$
Hence
$$ \|S:Y_{A_m(\tilde{\omega})}\to X_{B_m(\omega)}\|\ls \|S-T:Y_{A_m(\tilde{\omega})}\to X_{B_m(\omega)}\|+\|T:Y_{A_m(\tilde{\omega})}\to X_{B_m(\omega)}\|\ls \left(\gamma+bt\right)\sqrt n.$$
Thus $\tilde{\omega}\in Q_t\left(\omega,\gamma+bt,m\right)$ and the claim follows. 
\end{proof}

\medskip

Using Proposition~\ref{prop:entropy-first-model}, we obtain
\begin{align*}
\mathbb{P}(Q(\omega,\gamma,m))
&\ls \mathbb{P}\left(Q_t\left(\omega,\gamma+bt,m\right)\right)\\
&\ls \sum_{T\in \mathcal{N}_t(\omega)} \mathbb{P}\left(\left\{\tilde{\omega}\in \widetilde{\Omega}_0:\ \|T:Y_{A_m(\tilde{\omega})}\to X_{B_m(\omega)}\| \ls \left(\gamma+bt\right)\sqrt n\right\}\right)\\
&\ls |\mathcal{N}_t(\omega)|\, \sup_{T\in \mathcal{N}_t(\omega)} \mathbb{P}\left(\left\{\tilde{\omega}\in \widetilde{\Omega}_0:\ \|T:Y_{A_m(\tilde{\omega})}\to X_{B_m(\omega)}\|\ls \left(\gamma+bt\right)\sqrt n\right\}\right).
\end{align*}
Since $\mathcal{N}_t(\omega)\subseteq \mathcal{M}_{B_m(\omega)}\subseteq SL_n$, Proposition~\ref{prop:gluskin-lower-bound-4} gives
$$ \mathbb{P}(Q(\omega,\gamma,m)) \ls \left(\frac{C_4b\sqrt{\ln(1+m/n)}}{t}\right)^{n^2} \left(C_3b\left(\gamma+bt\right)L_{\mu}\sqrt{\ln(1+m/n)}\right)^{mn}.$$
Now, write $\gamma ={r_0}/{\sqrt{\ln (1+m/n)}}$, where $r_0$ is an absolute positive constant to be chosen small enough and choose $0<t<1$ so that $\gamma=bt$. 
Then
\begin{align*}
\mathbb{P}(Q(\omega,\gamma,m)) &\ls \left(\frac{C_4b^2\sqrt{\ln(1+m/n)}}{\gamma}\right)^{n^2} \left(2C_3b\gamma L_{\mu}\sqrt{\ln(1+m/n)}\right)^{mn}\\
&\ls \left(\frac{C_4^{n/m}b^{2n/m}(\ln(1+m/n))^{n/m}}{r_0^{n/m}}(2C_3bL_{\mu}r_0)\right)^{mn}\\
&=\left(2C_3C_4^{n/m}b^{1+2n/m}L_{\mu}(\ln(1+m/n))^{n/m}\,r_0^{1-n/m}\right)^{mn}.
\end{align*}
Since $m\gr 2n$ and $x\mapsto \big(\ln (1+x)\big)^{1/x}$ is bounded on $[2,\infty)$, using also the fact that $C_3,C_4,b$ are absolute constants and $L_{\mu}\ls C$, we see that we can choose a small absolute constant $r_0>0$ so that
$$2C_3C_4^{n/m}b^{1+2n/m}L_{\mu}(\ln(1+m/n))^{n/m}\,r_0^{1-n/m}\ls \frac{1}{2}.$$
Therefore, $\mathbb{P}(Q(\omega,\gamma,m)) \ls 2^{-mn}.$
Consequently,
$$ \mathbb{P}(G(\gamma,m)) = \int_{\Omega_0}\mathbb{P}(Q(\omega,\gamma,m))\,d\mathbb{P}(\omega)  \ls \sup\{\mathbb{P}(Q(\omega,\gamma,m)):\ \omega\in \Omega_0\} 
\ls \left(\frac{1}{2}\right)^{mn}.$$
Therefore, 
$$\mathbb{P}(H(\gamma,m)) \gr  1-2\exp(-c\sqrt n)-2\mathbb{P}(G(\gamma,m)) \gr  1-2\exp(-c\sqrt n)-2^{-n^2+1},$$
and hence $\mathbb{P}(H(\gamma,m))\to 1 $ as $n\to\infty,$ uniformly for all integers $m$ satisfying $2n\ls m\ls \exp(\sqrt n).$

This completes the proof with $\gamma ={r_0}/{\sqrt{\ln (1+m/n)}}.$
It follows that
$$d_{{\rm BM}}(Y_{A_m(\tilde{\omega})},X_{B_m(\omega)})\gr \gamma^2 n \gr \frac{cn}{\ln(1+m/n)}$$
for an absolute constant $c>0$, with probability tending to $1$ as $n\to\infty$, uniformly for all integers $m$ satisfying $2n\ls m\ls \exp(\sqrt n)$. 
This proves the theorem.
\end{proof}

\bigskip 

\noindent \textbf{\S~4.2. Pure random polytope model.}  
We consider the pure random polytope model in which $x_1,\ldots,x_m$ are independent random vectors, defined on some probability space $\Omega$ and distributed according to an isotropic log-concave probability measure $\mu$ on $\mathbb{R}^n$ and $B_m=B_m(\omega)$ is the centrally symmetric random polytope
$$ B_m:=B_m(\omega)={\rm absconv}\{x_1(\omega),\ldots,x_m(\omega)\},\qquad \omega\in\Omega . $$
As in the basis-enriched model, we denote by $X_{B_m}$ the normed space whose unit ball is $B_m$. 

We also consider an independent copy of this construction and denote by $Y_{A_m}$ the normed space whose unit ball is $A_m={\rm absconv}\{y_1,\ldots,y_m\}$. 
We write $\widetilde{\Omega}$ for the probability space corresponding to the sample $(y_1,\ldots,y_m)$.

\smallskip 

Our goal is to prove the following theorem.

\begin{theorem}[log-concave Gluskin's theorem]\label{th:log-concave-gluskin-b}
There exists an absolute constant $c_0\gr 1$ such that if $c_0n\ls m\ls\exp(\sqrt{n})$ and $x_1,\ldots,x_m$ and $y_1,\ldots ,y_m$ are independent random vectors in $\R^n$ with isotropic log-concave distribution $\mu$ then the random polytopes 
$$B_m={\rm absconv}\{x_1,\ldots ,x_m\}\quad\text{and}\quad A_m={\rm absconv}\{y_1,\ldots ,y_m\}$$ 
satisfy
$$d_{{\rm BM}}(Y_{A_m},X_{B_m})\gr cn/\ln (1+m/n)$$
where $c>0$ is an absolute constant, with probability tending to $1$ as $n\to\infty$, uniformly in $m$.
\end{theorem}

We start with a few basic geometric properties of the random polytope $B_m$ that will be needed for the distance estimate.
Let $s=\lceil c_1 n\rceil $, where $c_1>0$ is a sufficiently large absolute constant such that $q:=\ln(1+s/n)\gr 2$. 
Then, since $\mu$ is isotropic, we have $B_2^n=Z_2(\mu)\subseteq Z_q(\mu)$ by \eqref{eq:zq-mu-inclusions}.  
Hence,  \eqref{eq:DGT} implies that if $x_1,\ldots,x_s$ are independent random vectors in $\R^n$ with distribution $\mu$, then
$${\rm absconv}\{x_1,\ldots ,x_s\}\supseteq a\,B_2^n$$
with probability greater than $1-\exp(-c_2 s)$, where $a,c_2>0$ are absolute constants.
After increasing the constant $c_0$ in the theorem if necessary, we may assume that $m\gr 2s$. 
Hence $\ln(1+(m-s)/n)\gr q$ and therefore the same argument gives
$${\rm absconv}\{x_{s+1},\ldots,x_m\}\supseteq a\,B_2^n$$
with probability greater than $1-\exp(-c_2(m-s))$.
On the other hand, \eqref{eq:euclidean-norm} shows that if $m\ls \exp(\sqrt{n})$ then $|x_i|\ls b\sqrt{n}$ for all $i=1,\ldots ,m$ with probability greater than $1-\exp(-c\sqrt{n})$.
We summarize these facts in the next proposition.

\begin{proposition}[geometry of $B_m$]\label{prop:good-event}
Let $2s=2\lceil c_1 n\rceil \ls m\ls \exp(\sqrt{n})$. 
If $x_1,\ldots,x_m$ are independent random vectors in $\R^n$ with isotropic log-concave distribution $\mu$, then
$$a\,B_2^n\subseteq {\rm absconv}\{x_1,\ldots ,x_s\}\cap {\rm absconv}\{x_{s+1},\ldots ,x_m\}$$
and
$$B_m\subseteq b\sqrt{n}\,B_2^n $$
with probability greater than $1-\exp(-c\sqrt n)$, where the constants $b\gr 1$ and $a,c,c_1>0$ are absolute.
\end{proposition}

Set $B_s:=B_s(x_1,\ldots ,x_s)={\rm absconv}\{x_1,\ldots ,x_s\}$ and $A_s:=A_s(y_1,\ldots ,y_s)={\rm absconv}\{y_1,\ldots ,y_s\}$. 
We define the ``good events"
$$\Omega_0=\left\{\omega\in\Omega: a\,B_2^n\subset B_s(\omega) \text{ and }\ |x_j(\omega)|\ls b\sqrt n,  1\ls j\ls m \right\}, $$
and
$$\widetilde\Omega_0= \left\{\tilde{\omega}\in\widetilde{\Omega}: a\,B_2^n\subset A_s(\tilde{\omega})  \text{ and } |y_j(\tilde{\omega})|\ls b\sqrt n,\ \ 1\ls j\ls m \right\}, $$
where $a,b>0$ are the constants from Proposition~\ref{prop:good-event}. 
Then,
\begin{equation}\label{eq:good-events}
\mathbb P(\Omega_0)\gr 1-\exp(-c\sqrt n), \qquad \mathbb{P}(\widetilde\Omega_0)\gr 1-\exp(-c\sqrt n).
\end{equation}

We shall also use the fact that every linear operator $T:\mathbb{R}^n\to\mathbb{R}^n$ admits a singular value decomposition: there exist orthonormal bases $\{\bar u_i(T)\}_{i=1}^n$ and $\{u_i(T)\}_{i=1}^n$ of $\mathbb{R}^n$ such that
\begin{equation}\label{eq:polar-decomposition}
T=\sum_{i=1}^ns_i(T)u_i(T)\otimes \bar u_i(T)
\end{equation}
where $s_1(T)\gr s_2(T)\gr\cdots \gr s_n(T)\gr 0$ are the singular values of $T$. 
Equivalently,
$$T(x)=\sum_{i=1}^n s_i(T)\,\langle x,\bar u_i(T)\rangle\,u_i(T),\qquad x\in\mathbb{R}^n.$$
This representation is not unique in general. 
Nevertheless, the sequence $\{s_i(T)\}_{i=1}^n$ is uniquely determined by $T$.

\begin{proposition}\label{prop:fixed-operator}
Let $B_0={\rm absconv}\{z_1,\ldots ,z_m\}\subseteq b\sqrt{n}B_2^n$ and let $\gamma>0$.
For any $n/4\ls k\ls n$ and every $T\in M_n(\mathbb{R})$ with $s_k(T)\gr 1$ we have
$$\mu\left(\{x\in\mathbb{R}^n:T(x)\in \gamma\sqrt{n}B_0\}\right)\ls e^{-k},$$
provided that
$$\gamma \ls \frac{1}{C_5b\sqrt{\ln (1+m/n)}},$$
where $C_5>0$ is an absolute constant.
\end{proposition}

\begin{proof}
Let $ E={\rm span}\{u_i(T):1\ls i\ls k\}$ and $F={\rm span}\{\bar u_i(T):1\ls i\ls k\}$. 
Note that $T(F)=E$. 
It is clear that
$$\mu\left(\left\{x\in\mathbb R^n:T(x)\in \gamma \sqrt{n}\,B_0\right\}\right) \ls \mu\left(\left\{x\in\mathbb R^n:P_E(T(x))\in \gamma \sqrt{n}\,P_E(B_0)\right\}\right).$$
By the singular value decomposition \eqref{eq:polar-decomposition} we have $P_E(T(x))=T(P_F(x))$ for all $x\in\mathbb R^n$.
It follows that
\begin{align*}
\mu\left(\left\{x\in\mathbb R^n:P_E(T(x))\in \gamma \sqrt{n}\,P_E(B_0)\right\}\right)
&=\mu\left(\left\{x\in\mathbb R^n:T(P_F(x))\in \gamma\sqrt{n}\,P_E(B_0)\right\}\right)\\
&= \mu\left(\left\{x\in\mathbb R^n:P_F(x)\in \gamma \sqrt{n}\,(T|_F)^{-1}(P_E(B_0))\right\}\right).
\end{align*}
By definition of the marginal $\pi_F(\mu)$,
$$ \mu\left(\left\{x\in\mathbb R^n:P_Fx\in \gamma \sqrt{n}\,(T|_F)^{-1}(P_E(B_0))\right\}\right) = \pi_F(\mu)\big(\gamma \sqrt{n}\,(T|_F)^{-1}(P_E(B_0))\big). $$
Since $\mu$ is isotropic and log-concave, the marginal $\pi_F(\mu)$ is a $k$-dimensional isotropic log-concave probability measure. 
Therefore, by Lemma~\ref{lem:measure-vs-volume},
$$\pi_F(\mu)\big(\gamma \sqrt{n}\,(T|_F)^{-1}(P_E(B_0))\big) \ls (C_1\gamma\sqrt{n})^k\, \vol_k\big((T|_F)^{-1}(P_E(B_0))\big). $$
Since $s_k(T)\gr 1$, we have $\|(T|_F)^{-1}:E\to F\|_{{\rm op}}\ls 1$ and hence $ |\det(T|_F)|^{-1}\ls 1.$ 
Consequently,
$$ \vol_k\big((T|_F)^{-1}(P_E(B_0))\big) = |\det(T|_F)|^{-1}\vol_k(P_E(B_0)) \ls \vol_k(P_E(B_0)). $$
Combining the previous estimates, we obtain
$$ \mu\left(\left\{x:T(x)\in \gamma\sqrt{n}\,B_0\right\}\right) \ls (C_1\gamma\sqrt{n})^{k}\vol_k(P_E(B_0)).$$
It remains to estimate $\vol_k(P_E(B_0))$. 
We have assumed that $B_0\subseteq b\sqrt{n}\,B_2^n,$ and therefore $ P_E(B_0)\subseteq b\sqrt{n}\,B_E. $
Hence Lemma~\ref{lem:barany-furedi}, applied in the $k$-dimensional space $E$ gives
$$ \vol_k(P_E(B_0)) \ls \left(\frac{C_2b\sqrt{n}\sqrt{\ln(1+m/k)}}{\sqrt{k}}\right)^k\vol_k(B_E) \ls \left(C_1^{\prime}b\sqrt{n}\,\frac{\sqrt{\ln(1+m/k)}}{k}\right)^k,$$
where in the last inequality we used $\vol_k(B_E)^{1/k}\approx 1/\sqrt{k}$.
Therefore,
$$\mu\left(\left\{x:T(x)\in \gamma\sqrt{n}\,B_0\right\}\right) \ls \left( C_2^{\prime}b\gamma\,\frac{n}{k}\,\sqrt{\ln(1+m/k)} \right)^{k}. $$
Finally, since $k\gr n/4$, we have $m/k\ls 4m/n$ and therefore $\ln(1+m/k)\ls \ln(1+4m/n)\ls C_3^{\prime}\ln(1+m/n).$
Also, $n/k \ls 4$. 
Thus
$$\mu\left(\left\{x:T(x)\in \gamma\sqrt{n}\,B_0\right\}\right) \ls \left(C_4^{\prime}b\gamma\sqrt{\ln(1+m/n)}\right)^{k}\ls e^{-k} $$
provided that $\gamma\ls 1/(C_4^{\prime}eb\sqrt{\ln (1+m/n)})$.
This proves the proposition with $C_5=C_4^{\prime}e$.
\end{proof}

\begin{proof}[\textbf{Proof of Theorem~$\mathbf{\ref{th:log-concave-gluskin-b}}$}] 
Let $s=\lceil c_1n\rceil$, where $c_1$ is the absolute constant fixed above. 
After increasing the constant $c_0$ in the statement if necessary, we may assume that $m\gr 2s.$
Set
$$B_s(x):={\rm absconv}\{x_1,\ldots,x_s\}, \qquad A_s(y):={\rm absconv}\{y_1,\ldots,y_s\}. $$
We also write
$$ B_m(x):={\rm absconv}\{x_1,\ldots,x_m\}, \qquad A_m(y):={\rm absconv}\{y_1,\ldots,y_m\}. $$

Fix $\tilde\omega\in \widetilde\Omega_0$ and set $B_0:={\rm absconv}\left\{y_1(\tilde\omega),\ldots,y_m(\tilde\omega)\right\}$. 
Since $A_s(\tilde\omega)\subseteq B_0$ and $\tilde\omega\in \widetilde\Omega_0$, we have
\begin{equation}\label{eq:B0-contains-ball}
aB_2^n\subseteq B_0\subseteq b\sqrt{n}B_2^n.
\end{equation}
We fix
\begin{equation}\label{eq:gamma-choice}
\gamma:=\frac{1}{C_5b\sqrt{\ln(1+m/n)}},
\end{equation}
where $C_5>0$ is the absolute constant from Proposition~\ref{prop:fixed-operator}. 
Since $B_0\subseteq b\sqrt n\,B_2^n$, we have
\begin{equation}\label{eq:fixed-operator-bound}
\mu\left(\{x\in\mathbb{R}^n:T(x)\in \gamma \sqrt{n}B_0\}\right)\ls e^{-k}
\end{equation}
for every $n/4\ls k\ls n$ and every $T\in M_n(\mathbb{R})$ with $s_k(T)\gr 1$.

\medskip

\noindent \textbf{Step I. Estimate for a fixed operator.}
Let $k:=\lfloor n/2\rfloor$ and let $T\in M_n(\mathbb{R})$ satisfy $s_k(T)\gr 1.$ By \eqref{eq:fixed-operator-bound}, for every $j=s+1,\ldots,m$ we have
$\mathbb P\left(T(x_j)\in \gamma \sqrt n\,B_0\right)\ls \exp(-k).$
Since the vectors $x_{s+1},\ldots,x_m$ are independent, we obtain
\begin{equation}\label{eq:step-I}
\mathbb P\left(T(x_j)\in \gamma \sqrt n\,B_0,\ \ s+1\ls j\ls m\right) \ls \exp\big(-k(m-s)\big).
\end{equation}

\medskip

\noindent \textbf{Step II. A net in the set of operators.}
Fix a realization $\omega'$ of the first $s$ vectors and assume that
$$aB_2^n\subset B_s(x(\omega'))\qquad\text{and}\qquad |x_j(\omega')|\ls b\sqrt n,\ \ 1\ls j\ls s.$$
By Carath\'eodory's theorem,
$$B_s(x(\omega'))\subset \bigcup_{\sigma\subset [s], |\sigma|=n}{\rm absconv}\{x_j(\omega'):\ j\in \sigma\}. $$
Hence
$$\vol_n(B_s(x(\omega'))) \ls \binom{s}{n}\max_{\sigma\subset [s], |\sigma|=n}\vol_n({\rm absconv}\{x_j(\omega'):\ j\in \sigma\}). $$
Since $aB_2^n\subseteq B_s(x(\omega'))$ and $s=\lceil c_1n\rceil$, which implies that $\binom{s}{n}\ls\left(\frac{es}{n}\right)^n\ls C_6^n$ for an absolute constant $C_6>0$, 
there exists a set $\sigma_1\subset [s]$ with $|\sigma_1|=n$ such that
\begin{equation}\label{eq:sigma1}
\vol_n(B_2^n)\ls (C_6/a)^n\vol_n({\rm absconv}\{x_j(\omega'):\ j\in \sigma_1\}).
\end{equation}
In particular, the vectors $\{x_j(\omega'):\ j\in \sigma_1\}$ are linearly independent.

Define
$$ \mathcal A(\omega',\tilde{\omega}):= \left\{ T\in M_n(\mathbb{R}): T(x_j(\omega'))\in \gamma\sqrt n\,B_0\;  \text{for all}\; j\in \sigma_1 \right\}, $$
and
$$ \mathcal A_0(\omega',\tilde{\omega}):= \left\{ T\in M_n(\mathbb{R}): s_k(T)\gr 1 \;\;\text{and}\;\;T(x_j(\omega'))\in \gamma\sqrt n\,B_0  \;\text{for all}\; 1\ls j\ls s \right\}. $$
Then $\mathcal A_0(\omega',\tilde{\omega})\subset \mathcal A(\omega',\tilde{\omega}).$

We apply Proposition~\ref{prop:entropy-general} to the linearly independent vectors $x_j(\omega')$, $j\in \sigma_1$, with $B:=\gamma\sqrt n\,B_0$ and $r:=a\gamma\sqrt n.$
By \eqref{eq:B0-contains-ball}, we have $rB_2^n\subseteq B$.
Moreover, by the assumption on $\omega'$, we have $\max_{j\in\sigma_1}|x_j(\omega')|\ls b\sqrt n, $ and hence $\delta_0:=\big(\max_{j\in\sigma_1}|x_j(\omega')|\big)^{-1}\gr \frac{1}{b\sqrt n}.$
Therefore, $r\delta_0\gr a\gamma/b.$
Applying Proposition~\ref{prop:entropy-general} with $t=\frac{a\gamma}{2b},$ we obtain a $t$-net $\mathcal N(\omega',\tilde{\omega})\subset \mathcal A_0(\omega',\tilde{\omega})$ with respect to the operator norm such that
\begin{align}
|\mathcal N(\omega',\tilde{\omega})|
&\ls \left(\frac{2C_0b}{a\gamma}\right)^{n^2} \left( \frac{\vol_n(B)\,\vol_n(B_1^n)} {\vol_n(B_2^n)\,\vol_n(\operatorname{absconv}\{x_j(\omega'):j\in\sigma_1\})} \right)^n \label{eq:net1}\\
&\ls \left(\frac{2C_0C_6b}{a^2\gamma}\right)^{n^2} \left( \frac{\vol_n(B)\,\vol_n(B_1^n)} {\vol_n(B_2^n)^2} \right)^n, \nonumber
\end{align}
where in the last inequality we used \eqref{eq:sigma1}.

Recall that $\vol_n(B_2^n)^{1/n}\approx 1/\sqrt n$ and $\vol_n(B_1^n)^{1/n}\approx 1/n$.
By Lemma~\ref{lem:barany-furedi} and~\eqref{eq:B0-contains-ball},
$$ \vol_n(B) =(\gamma\sqrt n)^n\vol_n(B_0) \ls \big(C_2b\gamma\sqrt n\,\sqrt{\ln(1+m/n)}\big)^n\vol_n(B_2^n).$$
Substituting this estimate into \eqref{eq:net1}, we get
\begin{equation}\label{eq:net2}
|\mathcal N(\omega',\tilde{\omega})| \ls \left(C_7(b/a)^2\sqrt{\ln(1+m/n)}\right)^{n^2},
\end{equation}
where $C_7>0$ is an absolute constant.

\medskip

\noindent \textbf{Step III. Approximation argument.}
For fixed $\tilde{\omega}\in\widetilde\Omega_0$, define
$$\Theta_{\tilde{\omega}}:=\Bigl\{\omega\in \Omega_0:\ \exists\, T\in M_n(\mathbb{R})\ \text{with}\ s_k(T)\gr 1\ \text{and}\ \|T:X_{B_m(x(\omega))}\to Y_{A_m(y(\tilde{\omega}))}\|\ls \frac{\gamma\sqrt n}{2}\Bigr\}.$$
Fix an admissible realization $\omega'$ of the first $s$ vectors and let $\Theta_{\tilde{\omega}}(\omega')$ denote the fiber of $\Theta_{\tilde{\omega}}$ over this realization, namely
\begin{align*}
\Theta_{\tilde{\omega}}(\omega') := \Bigl\{ \omega\in \Omega_0:\ &x_j(\omega)=x_j(\omega')\ \text{for }1\ls j\ls s,\\ 
&\exists\, T\in M_n(\mathbb{R}) \text{with } s_k(T)\gr 1 \ \text{and }\ \|T:X_{B_m(x(\omega))}\to Y_{A_m(y(\tilde{\omega}))}\|\ls \frac{\gamma\sqrt n}{2} \Bigr\}.
\end{align*}
We estimate $\mathbb P(\Theta_{\tilde{\omega}}(\omega')\mid x_1,\ldots,x_s)$. Let $\mathcal N(\omega',\tilde{\omega})$ be the net constructed in Step II.
Suppose that $\omega\in \Theta_{\tilde{\omega}}(\omega')$ and choose $T$ satisfying the defining properties of $\Theta_{\tilde{\omega}}(\omega')$.
Then
$$ T(x_j(\omega))\in \frac{\gamma\sqrt n}{2}\,B_0, \qquad 1\ls j\ls m. $$
Since $x_j(\omega)=x_j(\omega')$ for $1\ls j\ls s$, it follows that $T(x_j(\omega'))\in \gamma\sqrt n\,B_0$ for all $1\ls j\ls s,$ and hence 
$T\in \mathcal A_0(\omega',\tilde{\omega})$. 

Choose $T_0\in \mathcal N(\omega',\tilde{\omega})$ such that $\|T-T_0\|_{\rm op}\ls \frac{a\gamma}{2b}$. 
Then, for every $s+1\ls j\ls m$,
\begin{align*}
\|T_0(x_j(\omega))\|_{B_0}
&\ls \|T(x_j(\omega))\|_{B_0}+\|(T_0-T)(x_j(\omega))\|_{B_0} \ls \frac{\gamma\sqrt n}{2}+\frac{|(T_0-T)(x_j(\omega))|}{a}\\
&\ls \frac{\gamma\sqrt n}{2}+\frac{\gamma}{2b}\,|x_j(\omega)| \ls \frac{\gamma\sqrt n}{2}+\frac{\gamma\sqrt n}{2}=\gamma\sqrt{n}.
\end{align*}
Therefore
\begin{align*}
\Theta_{\tilde{\omega}}(\omega') \subset \Bigl\{ \omega\in\Omega_0:\ &x_j(\omega)=x_j(\omega')\ \text{for }1\ls j\ls s,\\
&\exists\, S\in \mathcal N(\omega',\tilde{\omega}) \text{ such that } S(x_j(\omega))\in \gamma\sqrt n\,B_0,\ \ s+1\ls j\ls m \Bigr\}.
\end{align*}
Conditioning on the first $s$ vectors and using \eqref{eq:step-I} together with \eqref{eq:net2}, we obtain
\begin{align*}
\mathbb P\bigl(\Theta_{\tilde{\omega}}(\omega')\mid x_1,\ldots,x_s\bigr)
&\ls |\mathcal N(\omega',\tilde{\omega})|\, \exp\big(-k(m-s)\big)\\
&\ls \exp\!\left(n^2\ln\!\big(C_7(b/a)^2\sqrt{\ln(1+m/n)}\big)-k(m-s)\right).
\end{align*}

Recall that $m-s\gr \frac{m}{2}$ and $k=\lfloor n/2\rfloor\gr \frac{n}{3}.$
Moreover, if we assume that $m\gr c_0n$ for a large enough absolute constant $c_0>0$ then we have
$$\ln\!\big(C_7(b/a)^2\sqrt{\ln(1+m/n)}\big)\ls \frac{m}{12n}\ls \frac{m-s}{6n}.$$
Hence
$$ n^2\ln\!\big(C_7(b/a)^2\sqrt{\ln(1+m/n)}\big)-k(m-s) \ls \frac{1}{6}n(m-s)-\frac{1}{3}n(m-s) = -\frac{1}{6}n(m-s)\ls -c'n^2,$$
where $c'>0$ is an absolute constant.
Therefore, $\mathbb P\bigl(\Theta_{\tilde{\omega}}(\omega')\mid x_1,\ldots,x_s\bigr)\ls \exp(-c'n^2).$
Since this bound is uniform in every admissible realization $\omega'$, averaging over the first $s$ vectors gives
$$\mathbb P(\Theta_{\tilde{\omega}})\ls \exp(-c'n^2).$$
Integrating the last estimate over $\tilde{\omega}\in \widetilde\Omega_0$ and using \eqref{eq:good-events}, we get
\begin{equation}\label{eq:theta-complement}
\mathbb{P} \left(\left\{(\omega,\tilde{\omega})\in \Omega_0\times \widetilde\Omega_0:\ \omega\notin \Theta_{\tilde{\omega}}\right\}\right)
\gr 1-\exp(-c\sqrt n).
\end{equation}
By symmetry, the same estimate holds with the roles of $\omega$ and $\tilde{\omega}$ interchanged.
Consequently, for the set $\Xi$ of all pairs $(\omega,\tilde{\omega})\in \Omega_0\times \widetilde\Omega_0$ that satisfy
\begin{itemize}
\item[(i)] for every $T\in M_n(\mathbb{R})$ with $s_k(T)\gr1$, $\|T:X_{B_m(x(\omega))}\to Y_{A_m(y(\tilde{\omega}))}\|\gr \frac{\gamma\sqrt n}{2}$ and
\item[(ii)] for every $S\in M_n(\mathbb{R})$ with $s_k(S)\gr 1$, $\|S:Y_{A_m(y(\tilde{\omega}))}\to X_{B_m(x(\omega))}\|\gr \frac{\gamma\sqrt n}{2}$,
\end{itemize}
we see that
\begin{equation}\label{eq:final-good}
\mathbb{P}(\Xi)\gr 1-\exp(-c\sqrt n).
\end{equation}
Finally, fix $(\omega,\tilde{\omega})\in \Xi$ and let $T\in GL_n$. Set $\lambda:=\frac{1}{s_k(T)}.$
Then $ s_k(\lambda T)=1.$ Also, since $k=\lfloor n/2\rfloor$, we have $k\ls n-k+1$ and therefore $ s_k(T)\gr s_{n-k+1}(T). $
Hence $ s_k((\lambda T)^{-1}) = {s_k(T)}/{s_{n-k+1}(T)} \gr 1.$
It follows that
$$ \|\lambda T:X_{B_m(x(\omega))}\to Y_{A_m(y(\tilde{\omega}))}\| \gr \frac{\gamma\sqrt n}{2},
\quad \text{and} \quad
\|(\lambda T)^{-1}:Y_{A_m(y(\tilde{\omega}))}\to X_{B_m(x(\omega))}\| \gr \frac{\gamma\sqrt n}{2}. $$
Multiplying the two inequalities, we get
$$ \|T:X_{B_m(x(\omega))}\to Y_{A_m(y(\tilde{\omega}))}\|\, \|T^{-1}:Y_{A_m(y(\tilde{\omega}))}\to X_{B_m(x(\omega))}\| \gr \frac{\gamma^2 n}{4}. $$
Taking the infimum over all $T\in GL_n$, we conclude that
$$d_{{\rm BM}}(X_{B_m(x(\omega))},Y_{A_m(y(\tilde{\omega}))}) \gr \frac{\gamma^2 n}{4}\gr\frac{cn}{4\ln(1+m/n)}$$
for an absolute constant $c>0$, by our choice of $\gamma$ in \eqref{eq:gamma-choice}. By \eqref{eq:final-good},
$$\mathbb{P}\left(d_{{\rm BM}}(Y_{A_m},X_{B_m})\gr \frac{c\,n}{\ln(1+m/n)}\right)\gr 1-\exp(-c\sqrt n).$$
This proves the theorem. \end{proof}

\section{Distance to the class of unconditional spaces}\label{section:5}

Recall that an $n$-dimensional normed space $Y$ has a $1$-unconditional basis if there exists a basis $\{ e_1,\ldots ,e_n\}$ of $Y$ with the property
$$\Big\|\sum_{i=1}^n\lambda_ie_i\Big\|_Y =\Big\|\sum_{i=1}^n|\lambda_i|e_i\Big\|_Y$$
for every choice of real numbers $\lambda_1,\ldots ,\lambda_n$.

In this section we show that log-concave random spaces $X_{B_m}$ have the worst possible Banach--Mazur distance from the class $\mathcal{U}_n$ of $n$-dimensional normed spaces with a $1$-unconditional basis when $m\approx n$. 
We consider the pure random polytope model:
\begin{equation}\label{eq:unco-def-2}
B_m:={\rm absconv}\{x_1,\ldots, x_m\}
\end{equation}
where $x_1,\ldots,x_m$ are independent random vectors defined on some probability space $\Omega$,
which are distributed according to an isotropic log-concave probability measure $\mu$ on $\mathbb{R}^n$. 

The precise statement is the following.

\begin{theorem}\label{th:unco}
Let $\mu$ be an isotropic log-concave probability measure on $\mathbb{R}^n$. 
There exists an absolute constant $C>0$ such that if $0<\delta<1$ and $C\ln(\delta^{-1})n\ls m\ls\exp(\sqrt{n})$, then the random space $X_{B_m}$ defined by \eqref{eq:unco-def-2} satisfies with probability greater than $1-\delta $
$$d_{{\rm BM}}(X_{B_m},Y)\gr \frac{c\sqrt{n}}{\sqrt{\ln \left (1+m/n\right )}}$$
for every $n$-dimensional normed space $Y$ with a $1$-unconditional basis.
\end{theorem}

The proof is based on a series of lemmas. 
The first one is an analogue of Kashin's theorem \cite{Kashin-1977} in our context. 
However, our information on a random space $X_{B_m}$ allows us to obtain an estimate for the full range of values of $m$. 
A proof is given in \cite[Proposition~2.6]{Giannopoulos-VMilman-2000}.

\begin{lemma}\label{lem:concentration}
Let $\delta\in (0,1)$ and let $\mu$ be an isotropic log-concave probability measure on $\mathbb{R}^n$.
If $m\gr C\ln (\delta^{-1} )n$, then $m$ independent random points $x_1,\ldots ,x_m$ distributed according to $\mu$ satisfy with probability
greater than $1-\delta $ the inequality
$$c_1|y|\ls \frac{1}{m}\sum_{j=1}^m|\langle y,x_j\rangle |\ls c_2|y|$$
for all $y\in {\mathbb R}^n$, where $C,c_1,c_2>0$ are absolute constants.
\end{lemma}

Recall (see \cite{Diestel-Jarchow-Tonge-book} for a comprehensive account of absolutely summing operators) that the $1$-summing norm $\pi_1(T:Y^{\ast }\to\ell_2^n)$ of an operator $T:Y^{\ast }\to\ell_2^n$ is the minimum of all positive constants $\alpha $ with the following property: for every $m\in {\mathbb N}$ and every choice of vectors $z_1,\ldots ,z_m\in Y^{\ast }$
$$\sum_{j=1}^m|T(z_j)|\ls \alpha\sup_{y\in B_Y}\sum_{j=1}^m|\langle y,z_j\rangle |.$$
We start with the next lemma.

\begin{lemma}\label{lem:unco-1}
Let $x_1,\ldots,x_m\in\mathbb{R}^n$ satisfy the conclusion of Lemma~\ref{lem:concentration} and consider the normed space $X_{B_m}$. 
Then $\pi_1(I_n:X_{B_m}^{\ast}\to\ell_2^n)\approx 1$.
\end{lemma}

\begin{proof}
Let $z_1,\ldots ,z_N\in X_{B_m}^{\ast }$. 
Using Lemma~\ref{lem:concentration} we write
$$\sum_{j=1}^N|z_j|\ls\frac{1}{c_1}\cdot\frac{1}{m}\sum_{i=1}^m\sum_{j=1}^N|\langle x_i,z_j\rangle |
\ls\frac{1}{c_1}\cdot\sup_{y\in B_m}\sum_{j=1}^N|\langle y,z_j\rangle |.$$
This shows that $\pi_1(I_n)=\pi_1(I_n:X_{B_m}^{\ast}\to\ell_2^n)\ls c_1^{-1}$.

On the other hand, by the definition of $\pi_1(I_n)$ we also have
$$\frac{1}{m}\sum_{j=1}^m|x_j|\ls \pi_1(I_n)\sup_{y\in B_m}\frac{1}{m}\sum_{j=1}^m|\langle y,x_j\rangle |.$$
Then, the upper estimate from Lemma~\ref{lem:concentration} shows that
$$\sup_{y\in B_m}\frac{1}{m}\sum_{j=1}^m|\langle y,x_j\rangle |\ls c_2\sup_{y\in B_m}|y|.$$
From \eqref{eq:euclidean-norm} we know that with probability greater than $1-\exp(-c\sqrt{n})$ we have 
$$\varepsilon_0\sqrt{n}\ls |x_j|\ls b\sqrt{n} \qquad j=1,\ldots ,m$$ 
where $\varepsilon_0,b>0$ are absolute constants, so we also have $\sup_{y\in B_m}|y|\ls b\sqrt n$. 
Therefore we conclude that 
$$\varepsilon_0\sqrt{n}\ls \frac{1}{m}\sum_{j=1}^m|x_j|\ls c_2b\sqrt{n}\,\pi_1(I_n),$$
which completes the proof. 
\end{proof}

\smallskip

\begin{lemma}\label{lem:unco-2}
Let $Q$ be a parallelepiped contained in $B_m^{\circ }$. 
Then,
$$\vol_n(Q)^{1/n}\ls\frac{2\pi_1(I_n:X_{B_m}^{\ast }\to\ell_2^n)}{n}.$$
\end{lemma}

\begin{proof}This observation is due to Ball \cite{Ball-1991}. Let $S\in GL_n$ such that $S(B_{\infty }^n)=Q$.
Applying Hadamard's inequality, the arithmetic-geometric mean inequality and the definition of $\pi_1(S:\ell_{\infty }^n\to\ell_2^n)$, 
we get
\begin{align*}
\vol_n(Q)^{1/n} &= 2|\det S|^{1/n}\ls 2\left(\prod_{i=1}^n|S(e_i)|\right)^{1/n}\ls\frac{2}{n}\sum_{i=1}^n|S(e_i)|\\
&\ls \frac{2}{n}\pi_1(S:\ell_{\infty}^n\to\ell_2^n)\,\sup_{y\in B_1^n}\sum_{i=1}^n|\langle y,e_i\rangle| = \frac{2}{n}\pi_1(S:\ell_{\infty}^n\to\ell_2^n).
\end{align*}
Since
$$\pi_1(S:\ell_{\infty }^n\to\ell_2^n)\ls \| S:\ell_{\infty }^n\to X_{B_m}^{\ast }\|\, \pi_1(I_n:X_{B_m}^{\ast }\to\ell_2^n),$$
the result follows from the fact that $S(B_{\infty }^n)=Q\subseteq B_m^{\circ }$. 
\end{proof}

\smallskip

We can now complete the proof of Theorem~\ref{th:unco}. 

\begin{proof}[\textbf{Proof of Theorem~$\mathbf{\ref{th:unco}}$}] 
Let $Y$ be an $n$-dimensional normed space with a $1$-unconditional basis. 
We set $d=d_{{\rm BM}}(X_{B_m}^{\ast },Y)$ and may assume that 
$$dB_Y\supseteq B_m^{\circ }\supseteq B_Y.$$ 
Since $B_Y$ is a $1$-unconditional convex body, Losanovskii's theorem (see \cite[Chapter~3]{Pisier-book}) implies that there exists a parallelepiped $Q\subseteq B_Y$ such that $\vol_n(B_Y)/\vol_n(Q)\ls n^n/n!$.

From Lemmas~\ref{lem:unco-1} and \ref{lem:unco-2} we deduce that
$$\vol_n(B_m^{\circ })^{1/n}\ls d\,\vol_n(B_Y)^{1/n}\ls cd\,\vol_n(Q)^{1/n}\ls c_1d/n.$$
Then, \eqref{eq:Bm-polar-volume} implies that $d_{{\rm BM}}(X_{B_m}^{\ast },Y)=d\gr c_2\sqrt{n}/\sqrt{\ln (1+m/n)}$ for a random $B_m$, and the theorem follows by duality (recall that, in general, $d_{{\rm BM}}(X^{\ast},Y^{\ast})=d_{{\rm BM}}(X,Y)$ for any pair of $n$-dimensional normed spaces $X$ and $Y$).
\end{proof}

\begin{remark}\rm 
A special case of Theorem~\ref{th:unco} is that if $m\approx n$ then a random $B_m$ satisfies 
$$d_{{\rm BM}}(X_{B_m},Y)\gr c\sqrt{n}$$ for every $n$-dimensional normed space $Y$ with a 1-unconditional basis. 
In fact, this shows that if $m\approx n$ then a random $B_m$ satisfies
$$d_{{\rm BM}}(X_{B_m},\mathcal{U}_n)\approx \sqrt{n}$$
because $\ell_2^n\in\mathcal{U}_n$ and $d_{{\rm BM}}(X_{B_m},\ell_2^n)\ls\sqrt{n}$ by John's theorem.
\end{remark}

\section{Basis constant}\label{section:6}

Let $1\ls k\ls n/2$ and $\beta>0$. Recall that an operator $T\in M_n(\mathbb{R})$ is called $(k,\beta)$-mixing if there exists a subspace 
$E\subset \mathbb{R}^n$ with $\dim E\gr k$ such that
\begin{equation}\label{eq:mixing}
\operatorname{dist}(T(x),E)=|P_{E^{\perp}}T(x)|\gr\beta |x|
\end{equation}
for all $x\in E$. We denote by ${\rm Mix}_n(k,\beta)$ the set of all $(k,\beta)$-mixing operators in $M_n(\mathbb{R})$.

We shall use the following elementary properties of mixing operators (see~\cite{Mankiewicz-Tomczak-handbook}). 
An operator $T\in M_n(\mathbb{R})$ is $(k,1)$-mixing if and only if $T+\lambda I_n$ is $(k,1)$-mixing for every $\lambda\in\mathbb{R}$. 
Moreover, if $T\in {\rm Mix}_n(k,1)$ and $1\ls s\ls k$, then $T\in {\rm Mix}_n(s,1)$. 
Finally, if $T=2P$, where $P$ is a projection of rank $s$ (not necessarily orthogonal), then $T$ is $(\min\{s,n-s\},1)$-mixing.
Indeed, if $s\ls n/2$, then $P$ is $(s,1/2)$-mixing, hence $2P$ is $(s,1)$-mixing. If $s\gr n/2$, then $I_n-P$ has rank $n-s\ls n/2$ and is $(n-s,1/2)$-mixing, so $-2(I_n-P)$ is $(n-s,1)$-mixing, and therefore $2P=2I_n-2(I_n-P)$ is also $(n-s,1)$-mixing.

Consider the centrally symmetric random polytope
$$
B_m(\omega):={\rm absconv}\{x_1(\omega),\ldots, x_m(\omega)\}
$$
where $x_1,\ldots,x_m$ are independent random vectors on a probability space $\Omega$, 
distributed according to an isotropic log-concave probability measure $\mu$ on $\mathbb{R}^n$.
We denote by $X_{B_m}=X_{B_m(\omega)}$ the normed space whose unit ball is $B_m(\omega)$.

As in Section~\ref{section:4}, set $s=\lceil c_1n\rceil$, where $c_1>0$ is an absolute constant. 
Define
\begin{align}\label{eq:two-blocks}
\Omega_0 &:= \Big\{\omega\in\Omega:\ |x_j(\omega)|\ls b\sqrt{n}\;\;\hbox{for all}\;1\ls j\ls m\;\;\hbox{and}\;\; \\
\nonumber &\hspace*{2cm}a\,B_2^n\subseteq {\rm absconv}\{x_j(\omega)\}_{1\ls j\ls s}\cap {\rm absconv}\{x_j(\omega)\}_{s+1\ls j\ls m}\Big\}.
\end{align}
Then, for all $c_0n\ls m\ls \exp(\sqrt{n})$, where $c_0>0$ is a sufficiently large absolute constant (so that, in particular, $m\gr 2s$), we have
\begin{equation}\label{eq:omega-measure}
\mathbb{P}(\Omega_0)\gr  1-\exp(-c\sqrt{n}),
\end{equation}
for some absolute constants $a,b,c>0$. For every $1\ls j\ls m$, define
$$B_m^{(j)}(\omega)={\rm absconv}\{x_i(\omega):i\neq j\}.$$
Then $a\,B_2^n\subseteq B_m^{(j)}(\omega)$ for every $\omega\in\Omega_0$ and every $1\ls j\ls m$; indeed, if $j\ls s$ we use the second block in \eqref{eq:two-blocks}, while if $j>s$ we use the first one.

\smallskip 

Our aim is to prove the next theorem.

\begin{theorem}\label{th:mixing-1}
There exist absolute constants $C,c_1,c_2>0$ such that if $m\gr Cn^2(\ln n)$, then with probability greater than $1-\exp(-c_2\sqrt{n})$, we have
$$\|P:X_{B_m}\to X_{B_m}\|\gr\frac{c_1\sqrt{n}}{\sqrt{\ln(1+m/n)}}$$
for every projection $P:\mathbb{R}^n\to\mathbb{R}^n$ with rank $n/4\ls {\rm rank}(P)\ls 3n/4$. In particular, with the same probability, we have 
$${\rm bc}(X_{B_m})\gr c\sqrt{n}/\sqrt{\ln(1+m/n)}.$$
\end{theorem}

Fix $A={\rm absconv}\{y_1,\ldots ,y_m\}\subseteq b\sqrt{n}\,B_2^n$  and let $T\in M_n(\mathbb{R})$. 
For every $\gamma>0$, define
$$W(m,\gamma,A,T):=\Big\{\omega\in\Omega_0:\ \|T:X_{B_m(\omega)}\to Y_{A}\|\ls \gamma\sqrt{n}\Big\}.$$
The next proposition provides an upper bound for the probability of these events.

\begin{proposition}\label{prop:measure-of-W}
Let $A\subseteq b\sqrt{n}\,B_2^n$ and let $n/4\ls k\ls n$. 
For every $T\in M_n(\mathbb{R})$ with $s_k(T)\gr 1$ we have
$$\mathbb{P}(W(m,\gamma,A,T))=\mathbb{P}\left(\{\omega\in\Omega_0:T(B_m(\omega))\subseteq \gamma\sqrt{n}A\}\right)\ls e^{-km},$$
provided that $\gamma\ls \left(C_5b\sqrt{\ln(1+m/n)}\right)^{-1}$, where $C_5>0$ is an absolute constant.
\end{proposition}

\begin{proof}
Since $B_m(\omega)=\operatorname{absconv}\{x_1(\omega),\ldots,x_m(\omega)\}$, the inclusion $T(B_m(\omega))\subseteq \gamma\sqrt n\,A$ is equivalent to the assertion that $T(x_j(\omega))\in \gamma \sqrt{n}\,A$ for all $j=1,\ldots,m$. 
Therefore,
\begin{align*}
\mathbb{P}\left(\left\{\omega\in\Omega_0:T(B_m(\omega))\subseteq \gamma\sqrt n\,A\right\}\right)
&\ls \mathbb{P}\left(\left\{\omega:T(x_j(\omega))\in \gamma\sqrt n\,A\ \text{for all }j=1,\ldots,m\right\}\right)\\
&=\left(\mu\left(\left\{x\in\mathbb R^n:T(x)\in \gamma\sqrt n\,A\right\}\right)\right)^m,
\end{align*}
where we used the independence of the random vectors $x_1,\ldots,x_m$.

Applying Proposition~\ref{prop:fixed-operator} we obtain
$$\mu\left(\{x\in\mathbb{R}^n:T(x)\in \gamma\sqrt{n}A\}\right)\ls e^{-k},$$
provided that $\gamma \ls \left(C_5b\sqrt{\ln(1+m/n)}\right)^{-1},$ where $C_5>0$ is an absolute constant. Hence, if $\gamma$
satisfies this restriction, we get 
$$ \mathbb{P}\left(\left\{\omega\in\Omega_0:T(B_m(\omega))\subseteq \gamma\sqrt{n}\,A\right\}\right) \ls e^{-km}$$
as required.
\end{proof}

\begin{proposition}\label{prop:LO-6.2}
Let $m>4n$ and $1\ls k\ls n/2$ and set $\gamma =\left(2C_5b\sqrt{\ln (1+m/n)}\right)^{-1}$. 
For every $T\in {\rm Mix}_n(k,1)$ we have
$$\mathbb{P}\left(\{\omega\in\Omega_0:\|T:X_{B_m}\to X_{B_m}\|\ls\gamma\sqrt{n}\}\right) \ls \binom{m}{\ell}\left(\left(1+\frac{2b\sqrt{n}}{a}\right)\,e^{-k}\right)^{\ell},$$
where $\ell=\lceil m/(2n+3)\rceil$.
\end{proposition}

\begin{proof}For every $j=1,\ldots ,m$ we define
\begin{align*}
A_j 
&:=\{\omega\in\Omega_0:T(x_j)\in \gamma\sqrt{n}B_m\}=\{\omega\in\Omega_0:\exists |\lambda|\ls 1\;\hbox{s.t.}\;T(x_j)-\gamma\sqrt{n}\lambda x_j\in \gamma\sqrt{n}(1-|\lambda|)B_m^{(j)}\}\\
&\subseteq \{\omega\in\Omega_0:\exists |\lambda|\ls 1\;\hbox{s.t.}\,T(x_j)-\gamma\sqrt{n}\lambda x_j\in\gamma\sqrt{n}B_m^{(j)}\}.
\end{align*}
Let $\Lambda$ be a $a/(b\sqrt{n})$-net of the interval $[-1,1]$ with $|\Lambda|\ls 1+\frac{2b\sqrt{n}}{a}$. 
Since $a\,B_2^n\subseteq B_m^{(j)}(\omega)$ for every $\omega\in\Omega_0$, we have
\begin{equation} \label{eq:LO-6.4}
A_j \subseteq \bigcup_{\lambda\in\Lambda}\{\omega\in \Omega_0:T(x_j)-\gamma\sqrt{n}\lambda x_j\in 2\gamma\sqrt{n}B_m^{(j)}\}  \subseteq \bigcup_{\lambda\in\Lambda}\tilde A_{j,\lambda},
\end{equation}
where
$$\tilde{A}_{j,\lambda}:=\{\omega\in\Omega_0^{(j)}:T(x_j)-\gamma\sqrt{n}\lambda x_j\in 2\gamma\sqrt{n}B_m^{(j)}\}.$$
and
$$ \Omega_0^{(j)}:=\{\omega\in\Omega:|x_i(\omega)|\ls b\sqrt{n}\ \text{for}\ i\neq j\}\supseteq \Omega_0$$
Fix $\lambda\in\Lambda$ and set $S:=T-\gamma\sqrt{n}\lambda\,I_n$. Since $T\in {\rm Mix}_n(k,1)$, after passing to a $k$-dimensional subspace if necessary,  there exists a $k$-dimensional subspace
$E$ of $\mathbb{R}^n$ such that $|P_{E^{\perp}}T(x)|\gr |x|$ for every $x\in E$. It follows that
\begin{equation}\label{eq:LO-6.5}
|S(x)|\gr |P_{E^{\perp}}S(x)|=|P_{E^{\perp}}T(x)|\gr |x| \qquad\text{for all }x\in E.
\end{equation}
Since $E\in G_{n,k}$, the min-max characterization of singular values yields $s_k(S)\gr 1$.

\smallskip 

Let $\mathcal{F}_{\{j\}^c}$ denote the $\sigma$-algebra generated by $\{x_i:i\neq j\}$. 
Write $f$ for the indicator function of $\tilde{A}_{j,\lambda}$ and let $h=\mathbb{E}(f|\mathcal{F}_{\{j\}^c})$ be the conditional expectation of $f$ with respect to $\mathcal{F}_{\{j\}^c}$. 
Note that $\Omega_0^{(j)}$ is $\mathcal{F}_{\{j\}^c}$-measurable and
$\tilde{A}_{j,\lambda}\subseteq \Omega_0^{(j)}$. Therefore,
$$h(\omega)\equiv 0\;\;\hbox{if}\;\;\omega\notin \Omega_0^{(j)}.$$

On the other hand, if $\omega\in\Omega_0^{(j)}$, then $B_m^{(j)}(\omega)\subseteq b\sqrt{n}\,B_2^n.$
Hence, using \eqref{eq:LO-6.5}, the independence of $x_j$ from $\mathcal{F}_{\{j\}^c}$ and Proposition~\ref{prop:fixed-operator}, we obtain
$$h(\omega) =\mu\left(\big\{x\in\mathbb{R}^n:S(x)\in 2\gamma\sqrt{n}\,B_m^{(j)}(\omega)\big\}\right)\ls e^{-k},$$
provided that $2\gamma\ls \left(C_5b\sqrt{\ln(1+m/n)}\right)^{-1}.$
It follows that
\begin{equation}\label{eq:LO-6.6}
h(\omega)=\mathbb{P}(\tilde{A}_{j,\lambda}\mid\mathcal{F}_{\{j\}^c})(\omega)\ls e^{-k} \qquad \text{for all }\omega\in\Omega.
\end{equation}
Since $\lambda\in\Lambda$ was arbitrary, we get
\begin{equation}\label{eq:LO-6.7}
\mathbb{P}(\tilde{A}_{j,\lambda}\mid\mathcal{F}_{\{j\}^c})\ls e^{-k} \qquad \text{for every }\lambda\in\Lambda.
\end{equation}
Setting $\tilde{A}_j=\bigcup_{\lambda\in\Lambda}\tilde{A}_{j,\lambda}$, it follows from \eqref{eq:LO-6.7} that
\begin{equation}\label{eq:LO-6.8}
\mathbb{P}(\tilde{A}_j\mid\mathcal{F}_{\{j\}^c}) \ls \sum_{\lambda\in\Lambda}\mathbb{P}(\tilde{A}_{j,\lambda}\mid\mathcal{F}_{\{j\}^c})
\ls \left(1+\frac{2b\sqrt{n}}{a}\right)\,e^{-k}.
\end{equation}

We shall use the next theorem of Szarek and Tomczak-Jaegermann from \cite{Szarek-Tomczak-2007}, which provides a ``small ball" estimate for the
probability of the intersection of weakly dependent events.

\begin{theorem}\label{th:decoupling}Let $d,m\in\mathbb{N}$. Consider a family of events $\{\Theta_{j,D}:j\in [m], D\subseteq [m]\}$ such that
$$\Theta_{j,D}\subseteq \bigcup_{\substack{D^{\prime}\subseteq D\\ |D^{\prime}|\ls d}} \Theta_{j,D^{\prime}}.$$
For every $j\in [m]$ set $\Theta_j:=\Theta_{j,\{j\}^c}$ and for every $r\in [m]$ set $\mathcal{J}_r=\{J\subseteq [m]:|J|=r\}$.
Let $\{\mathcal{F}_D:D\subseteq [m]\}$ be a nested family of $\sigma$-algebras, i.e. $D^{\prime}\subseteq D$ implies $\mathcal{F}_{D^{\prime}}\subseteq \mathcal{F}_D$. Assume
that for any $I,J\subseteq [m]$ with $I\cap J=\varnothing$ the events $\{\Theta_{j,I}:j\in J\}$ are $\mathcal{F}_I$-conditionally independent
and that $\mathbb{P}(\Theta_j\mid \mathcal{F}_{\{j\}^c})\ls p_j$ for every $j\in [m]$. Then, for every $r\ls \lceil m/(2d+1)\rceil$, we have
$$\mathbb{P}\left(\bigcap_{j=1}^m\Theta_j\right)\ls \sum_{J\in \mathcal{J}_r}\prod_{j\in J}p_j.$$
\end{theorem}

With the convention $\operatorname{absconv}(\varnothing)=\{0\}$, we apply Theorem~\ref{th:decoupling} as follows. For every $1\ls j\ls m$ and every $D\subseteq [m]$ we consider the event
\begin{align*}
\Theta_{j,D} =  \bigcup_{\lambda\in\Lambda}  \Bigl\{ \omega\in\Omega &: |x_i(\omega)| \ls b\sqrt{n} \text{ for all }i\in D, \\ 
&\text{and }  T(x_j(\omega))-\gamma\sqrt{n}\lambda x_j(\omega)\in 2\gamma\sqrt{n}\,\operatorname{absconv}\{x_i(\omega):i\in D\} \Bigr\}.
\end{align*}
Carath\'eodory's theorem, applied to the set $\{0,\pm x_i:i\in D\}$, implies that, for every $D\subseteq [m]$,
$$ \Theta_{j,D}\subseteq \bigcup_{\substack{D^{\prime}\subseteq D\\ |D^{\prime}|\ls n+1}}\Theta_{j,D^{\prime}}.$$
For every $J\subseteq [m]$ we denote by $\mathcal{F}_J$ the $\sigma$-algebra generated by $\{x_i:i\in J\}$. 
It is clear that if $I,J\subseteq [m]$ and $I\cap J=\varnothing$, then the events $\{\Theta_{i,I}:i\in J\}$ are $\mathcal{F}_I$-conditionally independent.

Also, by \eqref{eq:LO-6.4} and the definition of $\tilde{A}_j$, we have
\begin{equation}\label{eq:LO-9}
A_j\subseteq \Theta_{j,\{j\}^c}= \tilde{A}_j
\end{equation}
for every $j\in [m]$. 
Therefore, \eqref{eq:LO-6.8} yields
\begin{equation}\label{eq:LO-10}
\mathbb{P}(\Theta_{j,\{j\}^c}\mid \mathcal{F}_{\{j\}^c}) = \mathbb{P}(\tilde{A}_j\mid \mathcal{F}_{\{j\}^c}) \ls \left(1+\frac{2b\sqrt{n}}{a}\right)\,e^{-k}.
\end{equation}
Applying Theorem~\ref{th:decoupling} with $d=n+1$, $p_j=\left(1+\frac{2b\sqrt{n}}{a}\right) e^{-k}$ and $\ell=\left\lceil \frac{m}{2n+3}\right\rceil,$ and using \eqref{eq:LO-9} and \eqref{eq:LO-10}, we obtain
\begin{equation}\label{eq:LO-11}
\mathbb{P}\left(\bigcap_{j=1}^mA_j\right) \ls \mathbb{P}\left(\bigcap_{j=1}^m\Theta_{j,\{j\}^c}\right) \ls \binom{m}{\ell}\left(\left(1+\frac{2b\sqrt{n}}{a}\right)\,e^{-k}\right)^{\ell}.
\end{equation}
Since
$$\{\omega\in\Omega_0:\|T:X_{B_m(\omega)}\to X_{B_m(\omega)}\|\ls\gamma\sqrt{n}\}=\bigcap_{j=1}^mA_j,$$
the proof is complete.
\end{proof}

\begin{theorem}\label{th:LO-6.1}
Let $c_0n^2(\ln n)\ls m\ls\exp(\sqrt{n})$, where $c_0>0$ is an absolute constant. For every $n/4\ls k\ls n/2$ and $T\in {\rm Mix}_n(k,1)$ we have
\begin{equation}\label{eq:LO-1}\|T:X_{B_m(\omega)}\to X_{B_m(\omega)}\|\gr \frac{c\sqrt{n}}{\sqrt{\ln(1+m/n)}}\end{equation}
with probability at least $1-\exp(-c'\sqrt{n})$, where $c,c'>0$ are absolute constants. 
\end{theorem}

\begin{proof}
Let $\Omega_0$ be defined as in \eqref{eq:two-blocks}. 
We also fix $\gamma =\left(2C_5b\sqrt{\ln (1+m/n)}\right)^{-1}$ so that Proposition~\ref{prop:LO-6.2} holds.  For every $\omega\in\Omega_0$ we have $a\,B_2^n\subseteq B_m(\omega)\subseteq b\sqrt{n}B_2^n$ and hence
\begin{equation}\label{eq:LO-12}\|T:X_{B_m(\omega)}\to X_{B_m(\omega)}\|\ls (b/a)\sqrt{n}\,\|T:\ell_2^n\to\ell_2^n\|\end{equation}
for every $T\in M_n(\mathbb{R})$. Similarly,
\begin{equation}\label{eq:LO-13}\|T:X_{B_m(\omega)}\to X_{B_m(\omega)}\|\gr \frac{a}{b\sqrt{n}}\,\|T:\ell_2^n\to\ell_2^n\|.\end{equation}
From \eqref{eq:LO-13} we see that if $T\in M_n(\mathbb{R})$ satisfies $\|T:\ell_2^n\to\ell_2^n\|\gr (b/a)\gamma n$ then 
$\|T:X_{B_m(\omega)}\to X_{B_m(\omega)}\|\gr \gamma\sqrt{n}$ for every $\omega\in\Omega_0$. 
Therefore, in order to prove the theorem it suffices to consider only operators $T\in M_n(\mathbb{R})$ with $\|T:\ell_2^n\to\ell_2^n\|\ls (b/a)\gamma n$.
We define
$$\mathcal{A}:=\{T\in {\rm Mix}_n(k,1):\|T:\ell_2^n\to\ell_2^n\|\ls (b/a)\gamma n\}.$$
Set $\delta\ls a\gamma/(2b)$ (for instance $\delta=a\gamma/(4b)$).
Then a standard argument in the spirit of Lemma~\ref{lem:standard-net} shows that there exists a $\delta$-net $\mathcal{N}$ in $\mathcal{A}$ with respect to the operator norm, such that
\begin{equation}\label{eq:LO-14}|\mathcal{N}|\ls (C_8(b/a)^2n)^{n^2}.\end{equation}

For every $T\in\mathcal{A}$ we define
$$\Theta_T:=\{\omega\in\Omega_0:\|T:X_{B_m(\omega)}\to X_{B_m(\omega)}\|\ls\gamma\sqrt{n}\}$$
and set
\begin{equation}\label{eq:LO-15}
\Theta =\Omega_0\setminus\bigcup_{T\in\mathcal{N}}\Theta_T.
\end{equation}

\begin{claim}\label{claim:LO-6.16}
Let $T\in {\rm Mix}_n(k,1)$. For every $\omega\in\Theta$ we have
\begin{equation}\label{eq:LO-16}
\|T:X_{B_m(\omega)}\to X_{B_m(\omega)}\|\gr \frac{\gamma}{2}\sqrt{n}.
\end{equation}
\end{claim}

\begin{proof}
Fix $T\in {\rm Mix}_n(k,1)$ and $\omega\in\Theta$. If $T\notin\mathcal{A}$, then \eqref{eq:LO-16} holds, even with $\gamma$ in place of $\gamma/2$.
Thus, we may assume that $T\in\mathcal{A}$. Choose $T_0\in\mathcal{N}$ so that $ \|T-T_0:\ell_2^n\to\ell_2^n\|\ls \delta.$
Since $\omega\in\Theta$, we have $\omega\notin\Theta_{T_0}$ and therefore
$$\|T_0:X_{B_m(\omega)}\to X_{B_m(\omega)}\|>\gamma\sqrt{n}.$$
Since $B_m(\omega)=\operatorname{absconv}\{x_1(\omega),\ldots,x_m(\omega)\}$, there exists $j\in [m]$ such that $\|T_0(x_j(\omega))\|_{B_m(\omega)}>\gamma\sqrt{n};$ 
otherwise we would have $T_0(B_m(\omega))\subseteq \gamma\sqrt{n}\,B_m(\omega)$, contradicting $\omega\notin\Theta_{T_0}$.

As $\|x_j(\omega)\|_{B_m(\omega)}\ls 1$, we obtain
\begin{align}
\label{eq:LO-17}
\|T:X_{B_m(\omega)}\to X_{B_m(\omega)}\|
&\gr \|T(x_j(\omega))\|_{B_m(\omega)} \nonumber \gr \|T_0(x_j(\omega))\|_{B_m(\omega)}-\|(T-T_0)(x_j(\omega))\|_{B_m(\omega)} \nonumber\\
&\gr \gamma\sqrt{n}-\|T-T_0:X_{B_m(\omega)}\to X_{B_m(\omega)}\|\,\|x_j\|_{B_m(\omega)} \nonumber\\
&\gr \gamma\sqrt{n}-\delta (b/a)\sqrt{n}\gr \frac{\gamma}{2}\sqrt{n}.
\end{align}
where in the last step we used \eqref{eq:LO-12} and the choice $\delta\ls a\gamma/(2b)$.
This proves the claim.
\end{proof}

Having proved Claim~\ref{claim:LO-6.16}, it remains to estimate the probability of $\Theta$. 
By \eqref{eq:LO-15},
$$ \mathbb{P}(\Theta)\gr \mathbb{P}(\Omega_0)-\sum_{T\in\mathcal{N}}\mathbb{P}(\Theta_T). $$
Since $\mathcal N\subseteq \mathcal A\subseteq {\rm Mix}_n(k,1)$, Proposition~\ref{prop:LO-6.2} and \eqref{eq:LO-14} yield
$$ \mathbb{P}(\Theta)\gr \mathbb{P}(\Omega_0)-(C_8(b/a)^2n)^{n^2}\binom{m}{\ell}\left(\left(1+\frac{2b \sqrt{n}}{a}\right)e^{-k}\right)^{\ell}. $$
Using the fact that $\binom{m}{\ell}\ls \left(\frac{em}{\ell}\right)^{\ell} \ls (C_9n)^{\ell}$ because $\ell=\left\lceil \frac{m}{2n+3}\right\rceil$ and since $k\gr n/4,$
we get
$$(C_8(b/a)^2n)^{n^2}\binom{m}{\ell}\left(\left(1+\frac{2b\sqrt{n}}{a}\right)e^{-k}\right)^{\ell}\ls \exp\left(C_9n^2\ln n-n\ell /4\right).$$
Since $n\ell \gr m/5$, this last quantity is at most $\exp(-m/10)$ provided that $ m\gr 10C_9n^2(\ln n) $. 
Recalling that $ \mathbb{P}(\Omega_0)\gr 1-\exp(-c\sqrt n),$ we conclude that $ \mathbb{P}(\Theta)\gr 1-\exp(-c'\sqrt n).$
Claim~\ref{claim:LO-6.16} now, together with the choice of $\gamma$, gives \eqref{eq:LO-1}.
\end{proof}

We are now able to prove Theorem~\ref{th:mixing-1}.

\begin{proof}[\textbf{Proof of Theorem~$\mathbf{\ref{th:mixing-1}}$}]
Let $P$ be a projection with $n/4\ls {\rm rank}(P)\ls 3n/4$. 
Then, the operator $2P$ is $(n/4,1)$-mixing.
Hence, Theorem~\ref{th:LO-6.1} shows that
\begin{equation}\label{eq:final}\|2P:X_{B_m(\omega)}\to X_{B_m(\omega)}\|\gr \frac{c_1\sqrt{n}}{\sqrt{\ln(1+m/n)}}\end{equation}
with probability at least $1-\exp(-c_2\sqrt{n})$, uniformly for all such projections $P$, where $c_1,c_2>0$ are absolute constants. 
The main claim of the theorem follows after dividing the constant in the norm estimate by $2$. The lower bound for ${\rm bc}(X)$
is a consequence of \eqref{eq:low-bc}.
\end{proof}

\bigskip

\noindent {\bf Acknowledgement.} The second named author acknowledges support by a PhD scholarship from the National Technical University of Athens.

\bigskip 


\footnotesize
\bibliographystyle{amsplain}

\begin{thebibliography}{100}

\bibitem{AGA-book} \textrm{S.\ Artstein-Avidan, A.\ Giannopoulos and V.\ D.\ Milman},
\textit{Asymptotic Geometric Analysis, Vol. I}, Mathematical Surveys and Monographs, 202. American Mathematical Society, Providence, RI, 2015. xx+451 pp.
\bibitem{AGA-book-2} \textrm{S.\ Artstein-Avidan, A.\ Giannopoulos and V.\ D.\ Milman}, \textit{Asymptotic Geometric Analysis, Vol. II}, Mathematical Surveys and Monographs, 261. American  Mathematical Society, Providence, RI, 2021. xxxvii+645 pp.
\bibitem{Ball-1991} \textrm{K.\ M.\ Ball}, \textit{Normed spaces with a weak Gordon-Lewis property}, Functional analysis (Austin, TX, 1987/1989), 36--47,  Lecture Notes in Math., 1470, Longhorn Notes, Springer, Berlin, 1991.
\bibitem{Ball-Pajor-1990} \textrm{K.\ M.\ Ball and A.\ Pajor}, \textit{Convex bodies with few faces}, Proc. Amer. Math. Soc. 110 (1990), no.~1, 225--231.      
\bibitem{Barany-Furedi-1987} \textrm{I.\ B\'{a}r\'{a}ny and Z.\ F\"{u}redi}, \textit{Computing the volume is difficult}, Discrete Comput. Geom. 2 (1987), no.~4, 319--326. 
\bibitem{Benyamini-Gordon-1981} \textrm{Y.\ Benyamini and Y.\ Gordon}, \textit{Random factorization of operators between Banach spaces}, J. Analyse Math. 39 (1981), 45--74.
\bibitem{Bizeul-2025} \textrm{P.\ Bizeul}, \textit{The slicing conjecture via small ball estimates}, Preprint ({\tt https://arxiv.org/abs/2501.06854}).
\bibitem{Borell-1974} \textrm{C.\ Borell}, \textit{Convex measures on locally convex spaces}, Ark. Mat. {12} (1974), 239--252.
\bibitem{Bourgain-1986} \textrm{J.\ Bourgain}, \textit{On high dimensional maximal functions associated to convex bodies}, Amer. J. Math. 108 (1986), no.~6, 1467--1476. 
\bibitem{BGVV-book} \textrm{S.\ Brazitikos, A.\ Giannopoulos, P.\ Valettas and B-H.\ Vritsiou}, \textit{Geometry of isotropic convex bodies}, Mathematical Surveys and Monographs, 196. American Mathematical Society, Providence, RI, 2014. xx+594 pp.
\bibitem{Dafnis-Giannopoulos-Tsolomitis-2009} \textrm{N.\ Dafnis, A.\ Giannopoulos and A.\ Tsolomitis}, \textit{Asymptotic shape of a random polytope in a convex body},  J. Funct. Anal. 257 (2009), no.~9, 2820--2839. 
\bibitem{Diestel-Jarchow-Tonge-book} \textrm{J.\ Diestel, H.\ Jarchow and A.\ Tonge}, \textit{Absolutely summing operators}, Cambridge Studies in Advanced Mathematics, 43. Cambridge University Press, Cambridge, 1995. xvi+474 pp.
\bibitem{Giannopoulos-Hartzoulaki-2002} \textrm{A.\ Giannopoulos and M.\ Hartzoulaki}, \textit{Random spaces generated by vertices of the cube}, Discrete Comput. Geom. 28 (2002), no.~2, 255--273.
\bibitem{Giannopoulos-VMilman-2000} \textrm{A.\  Giannopoulos and V.\ D.\ Milman}, \textit{Concentration property on probability spaces}, Adv. Math. 156 (2000), no.~1, 77--106.
\bibitem{Giannopoulos-Pafis-Tziotziou-2025} \textrm{A.\ Giannopoulos, M.\ Pafis and N.\ Tziotziou}, \textit{The isotropic constant in the theory of high-dimensional convex bodies}, Bull. Hellenic Math. Soc. 69 (2025), 89--188.
\bibitem{Gluskin-1981} \textrm{E.\ D.\ Gluskin}, \textit{The diameter of the Minkowski compactum is approximately equal to $n$}, Funktsional. Anal. i Prilozhen. 15 (1981), no.~1, 72--73.
\bibitem{Gluskin-1981b} \textrm{E.\ D.\ Gluskin}, \textit{Finite dimensional analogues of spaces without basis}, Dokl. Akad. Nauk SSSR 261 (1981), no.~5, 1046--1050. 
\bibitem{Gluskin-1989}\textrm{E.\ D.\ Gluskin}, \textit{Extremal properties of orthogonal parallelepipeds and their applications to the geometry of Banach spaces}, Math. USSR Sbornik {64} (1989), no.~1, 85--96.
\bibitem{Guan-preprint} Q.~Y.~Guan, \textit{A note on Bourgain's slicing problem}, Preprint. ({\tt https://arxiv.org/abs/2412.09075}).
\bibitem{John-1948} \textrm{F.\ John}, \textit{Extremum problems with inequalities as subsidiary conditions}, Studies and Essays Presented to R. Courant on his 60th Birthday, January 8, 1948, 187--204, Interscience Publishers, New York, 1948.
\bibitem{Kashin-1977} \textrm{B.\ S.\ Kashin}, \textit{Sections of some finite-dimensional sets and classes of smooth functions}, Izv. Akad. Nauk. SSSR Ser. Mat. {41} (1977), no.~2, 334--351.
\bibitem{Klartag-Lehec-2025} B.~Klartag and J.~Lehec, \textit{Affirmative resolution of Bourgain's slicing problem using Guan's bound}, Geom. Funct. Anal. 35 (2025), no.~4, 1147--1168.
\bibitem{LMOTJ-2007} \textrm{R.\ Lata{\l}a, P.\ Mankiewicz, K.\ Oleszkiewicz and N.\ Tomczak-Jaegermann}, \textit{Banach--Mazur distances and projections on random subgaussian polytopes}, Discrete Comput. Geom. 38 (2007), no.~1, 29--50.
\bibitem{Litvak-Pajor-Rudelson-Tomczak-2005} \textrm{A.\ E.\ Litvak, A.\ Pajor, M.\ Rudelson and N.\ Tomczak-Jaegermann}, \textit{Smallest singular value of random matrices and geometry of random polytopes},   Adv. Math. 195 (2005), no.~2, 491--523.
\bibitem{Lutwak-Yang-Zhang-2000} \textrm{E.\ Lutwak, D.\ Yang and G.\ Zhang}, \textit{$L_p$ affine isoperimetric inequalities}, Differential Geom. 56 (2000), no.~1, 111--132. 
\bibitem{Mankiewicz-1984} \textrm{P.\ Mankiewicz}, \textit{Finite dimensional spaces with symmetry constant of order $\sqrt{n}$}, Studia Math. 79 (1984), no.~2, 193--200.
\bibitem{Mankiewicz-Tomczak-handbook} \textrm{P.\ Mankiewicz and N.\ Tomczak-Jaegermann}, \textit{Quotients of finite-dimensional Banach spaces; random phenomena}, Handbook of the geometry of Banach spaces, Vol.~2, 1201--1246, North-Holland, Amsterdam, 2003. 
\bibitem{MT1} \textrm{P.\ Mankiewicz and N.\ Tomczak-Jaegermann}, \textit{Geometry of families of random projections of symmetric convex bodies}, Geom. Funct. Anal. 11 (2001), no.~6, 1282--1326.
\bibitem{Mendelson-Pajor-Rudelson-2005} \textrm{S.\ Mendelson, A.\ Pajor and M.\ Rudelson},  \textit{The geometry of random $\{-1,1\}$-polytopes}, Discrete Comput. Geom. 34 (2005), no.~3, 365--379.
\bibitem{Paouris-2006} \textrm{G.\ Paouris}, \textit{Concentration of mass in convex bodies},  Geom. Funct. Anal. 16 (2006), no.~5, 1021--1049. 
\bibitem{Paouris-2012} \textrm{G.\ Paouris}, \textit{Small ball probability estimates for log-concave measures}, Trans. Amer. Math. Soc. 364 (2012), no.~1, 287--308.
\bibitem{Pisier-book} \textrm{G.\ Pisier}, \textit{The Volume of Convex Bodies and Banach Space Geometry},  Cambridge Tracts in Mathematics, 94. Cambridge University Press, Cambridge, 1989. xvi+250 pp.
\bibitem{Szarek-1983} \textrm{S.\ J.\ Szarek}, \textit{The finite dimensional basis problem, with an appendix on nets of Grassman manifold}, Acta Math. 151 (1983), no.~3-4, 153--179. 
\bibitem{Szarek-1986a} \textrm{S.\ J.\ Szarek}, \textit{On the existence and uniqueness of complex structure and spaces with ``few" operators}, Trans. Amer. Math. Soc. 293 (1986), no.~1, 339--353.
\bibitem{Szarek-Tomczak-2007} \textrm{S.\ J.\ Szarek and N.\ Tomczak-Jaegermann}, \textit{Decoupling weakly dependent events}, Geometric aspects of functional analysis, 297--303, Lecture Notes in Math., 1910, Springer, Berlin, 2007. 
\bibitem{Tomczak-book} \textrm{N.\ Tomczak-Jaegermann}, \textit{Banach--Mazur Distances and Finite Dimensional Operator Ideals}, Pitman Monographs and Surveys in Pure and Applied Mathematics, 38. Longman Scientific \& Technical, Harlow; copublished in the United States with John Wiley \& Sons, Inc., New York, 1989. xii+395 pp. 
\end{thebibliography}

\bigskip

\thanks{\noindent {\bf Keywords:} Banach--Mazur distance; Gluskin spaces; Random polytopes; Isotropic log-concave probability measures.}

\smallskip

\thanks{\noindent {\bf 2020 MSC:} Primary 46B06; Secondary 46B20, 52A40, 52A23, 60D05.}

\bigskip

\bigskip 

\medskip 

\noindent \textsc{Apostolos \ Giannopoulos}: School of Applied Mathematical and Physical Sciences, National Technical University of Athens, Department of Mathematics, Zografou Campus, GR-157 80, Athens, Greece.

\smallskip

\noindent \textit{E-mail:} \texttt{apgiannop@math.ntua.gr}

\bigskip

\noindent \textsc{Antonios \ Hmadi}: School of Applied Mathematical and Physical Sciences, National Technical University of Athens, Department of Mathematics, Zografou Campus, GR-157 80, Athens, Greece.

\smallskip

\noindent \textit{E-mail:} \texttt{ahmadi@mail.ntua.gr}

\end{document}